\scrollmode
\documentclass[12pt]{amsart}

\usepackage{amssymb,amsmath,amsthm}
\usepackage{amsrefs}
\usepackage[all,arc]{xy}
\usepackage[nointegrals]{wasysym}
\usepackage{graphicx}
\usepackage{mathrsfs}
\usepackage{tikz}
\usepackage{tikz-cd}
\usepackage{multicol}
\usepackage{fancyhdr}
\usepackage{cite}
\usepackage[colorlinks=true, citecolor=red, linkcolor=red, urlcolor=magenta]{hyperref}

\DeclareMathOperator{\cd}{cd}

\DeclareMathOperator{\Gal}{Gal}

\DeclareMathOperator{\ad}{ad}

\DeclareMathOperator{\frk}{frk}

\DeclareFontFamily{U}{wncy}{}
\DeclareFontShape{U}{wncy}{m}{n}{<->wncyr10}{}
\DeclareSymbolFont{mcy}{U}{wncy}{m}{n}
\DeclareMathSymbol{\Sha}{\mathord}{mcy}{"58}
\DeclareMathSymbol{\sha}{\mathord}{mcy}{"78}

\newcommand{\euV}{\mathscr{V}}
\newcommand{\euE}{\mathscr{E}}

\newcommand{\K}{\mathbb{K}}

\newcommand{\F}{\mathbb{F}}

\newcommand{\Z}{\mathbb{Z}}

\newcommand{\spn}{\operatorname{Span}}
\renewcommand{\q}{\operatorname{q}\!}

\newcommand{\mc}[1]{\mathscr{#1}}
\newcommand{\bu}{\bullet}

\newcommand{\rad}{\operatorname{rad}}
\newcommand{\set}[2]{\left\lbrace{#1}\ \big\vert\ {#2}\right\rbrace}
\newcommand{\pres}[2]{\left\langle{#1}\, \big\vert\, {#2}\right\rangle}

\newcommand{\ul}{{\mc U(\mc L)}}
\newcommand{\li}{{\mc L}}
\newcommand{\bbu}{{\bu,\bu}}
\newcommand{\ext}{\operatorname{Ext}}
\newcommand{\argu}{\hbox to 1.5ex{\hrulefill}}

\newcommand{\exa}{{\Lambda_\bullet}}
\renewcommand{\hom}{\operatorname{Hom}}

\newcommand{\gr}{\operatorname{gr}}
\newcommand{\gen}[1]{\langle{#1}\rangle}
\newcommand{\mf}[1]{\mathfrak{#1}}

\newcommand{\hnn}{\operatorname{HNN}}

\newcommand{\rul}{{\underline{u}(\li)}}
\newcommand{\mi}{\mc M}

\usepackage{geometry}
\title{Koszul Lie Algebras and their subalgebras}

\author{S. Blumer}
\address{Fakultät für Mathematik, Universität Wien, Oskar-Morgenstern-Platz 1, 1090
	Wien, Austria}
\email{simone.blumer@univie.ac.at}

\begin{document}
	\maketitle

	
	
	\newtheorem{thm}{Theorem}[section]
	\newtheorem*{thmA}{Theorem A}
	\newtheorem*{thmB}{Theorem B}
	\newtheorem*{thmC}{Theorem C}
		\newtheorem*{thmD}{Theorem D}
	\newtheorem*{questionA}{Question A}
	\newtheorem*{corA}{Corollary A}
	\newtheorem*{thm*}{Theorem}
	\newtheorem{cor}[thm]{Corollary}
	\newtheorem{lem}[thm]{Lemma}
	\newtheorem{prop}[thm]{Proposition}
	\newtheorem{defin}[thm]{Definition}
	\newtheorem{exam}[thm]{Example}
	\newtheorem{examples}[thm]{Examples}
	\newtheorem{rem}[thm]{Remark}
	\newtheorem{case}{\sl Case}
	\newtheorem{claim}{Claim}
	\newtheorem{fact}[thm]{Fact}
	\newtheorem{question}[thm]{Question}
	\newtheorem*{questionss}{Questions}
	\newtheorem{conj}[thm]{Conjecture}
	\newtheorem*{notation}{Notation}
	\swapnumbers
	\newtheorem{rems}[thm]{Remarks}
	\newtheorem*{acknowledgment}{Acknowledgment}

	\numberwithin{equation}{section}

\begin{abstract}
	This paper examines (restricted) Koszul Lie algebras, a class of positively graded Lie algebras with a quadratic presentation and specific cohomological properties. The study employs HNN-extensions as a key tool for decomposing and analysing these algebras. 
	
	Building on a previous work on Koszul Lie algebras \cite{sb}, this paper also deals with Bloch-Kato Lie algebras, which constitute a distinguished subclass of that of Koszul Lie algebras where all subalgebras generated by elements of degree $1$ have a quadratic presentation. 
	It is shown that Bloch-Kato Lie algebras satisfy a version of the Levi decomposition theorem and that they satisfy the Toral Rank Conjecture. Two new families of such Lie algebras are introduced, including all graded Lie algebras generated in degree $1$ and defined by two quadratic relations.
	
	Throughout the paper, we show many properties of right-angled Artin graded (RAAG) Lie algebras, which form a large class of Koszul Lie algebras.
\end{abstract}

\section*{Introduction}
The notion of a Koszul algebra was introduced by Priddy \cite{priddy} for non-negatively graded algebras and found applications in various mathematical disciplines, including algebraic geometry (\hspace{1sp}\cite{conca},\hspace{1sp}\cite{wonderful}), representation theory (\hspace{1sp}\cite{bgs}), and combinatorial algebra (\hspace{1sp}\cite{frob},\cite{sheltyuz}).
By definition, a graded algebra $A=\bigoplus_{i\geq 0}A_i$ over a field $k$ with $A_0=k\cdot 1_A$ is Koszul if the trivial $A$-module $k$ admits a linear free $A$-resolution. In particular, it can be given a presentation where generators and relations have degree $1$ and $2$ respectively, i.e., it is a quadratic algebra.
A remarkable fact of such algebras is that their cohomology rings are quadratic as well, and they can thus be easily computed in terms of generators and relations. In fact, there is a duality between Koszul algebras and their cohomology rings, which are also Koszul, providing a powerful tool for studying their properties. This phenomenon -- generalizing the well-known duality between the symmetric and exterior algebras over a vector space -- led Kempf \cite{wonderful} to use the term ``wonderful rings" for designating these algebras.

Recently, Koszul algebras also appeared in the context of Galois cohomology (\hspace{1pt}{\cite{posit3}},\cite{posit2}) for it has been conjectured by Positselski that the Galois $\F_p$-cohomology ring of an absolute Galois group $G_\K=\Gal(\K_s/\K)$ of a field $\K$ containing a primitive $p$th root of $1$ is a Koszul algebra. 
In fact, by the affirmative answer to the Bloch-Kato conjecture, such a cohomology (as well as that of any closed subgroup $H$ of $G_\K$) is isomorphic with the $p$-reduced Milnor $K$-theory of the field $\K$ (resp. of the fixed field $\K_s^H$) via the Galois symbol, and hence it is quadratic. Pro-$p$ groups exhibiting this hereditary property in cohomology are known as Bloch-Kato pro-$p$ groups (\hspace{1sp}\cite{BKprop}). 

 A significant strengthening of Positselki's Koszulity conjecture was then proposed by Minac et al. in \cite{MPPT} predicting that the cohomology of such groups is not just Koszul but universally Koszul, an enhanced version of Koszulity introduced (in the commutative setting) by Conca \cite{conca}.

If $G$ is a finitely generated pro-$p$ group, then by a result of Lazard's \cite{lazard}, the universal restricted enveloping algebra of the restricted Lie algebra associated to the Jennings-Zassenhaus filtration of $G$ is naturally isomorphic to the associated graded algebra $\gr\F_p[G]$ of $\F_p[G]$ with respect to its augmentation ideal (see also \cite{analyticprop}). 
Moreover, a spectral sequence discovered by May \cite{may} relates the $\F_p$-cohomology of $\gr\F_p[G]$ to the Galois cohomology of $G$. When $\gr\F_p[G]$ is a Koszul algebra, the spectral sequence collapses at the first page, leading to an isomorphism $H^\bu(\gr\F_p[G],\F_p)\overset{\sim}{\longrightarrow}H^\bu(G,\F_p)$ (see \cite{weigelcohomology}). 

Weigel asked in \cite{weig2} if $\gr\F_p[G]$ is Koszul whenever $G$ is the maximal pro-$p$ quotient of the absolute Galois group of a field containing a primitive $p$th root of $1$. Motivated by this question and by the Bloch-Kato conjecture, we were thus led to study quadratically defined ($p$-restricted) Lie algebras such that their $\F_p$-cohomology, as well as those of some of their subalgebras, is Koszul.
Lie algebras of this kind were introduced in the author's paper \cite{sb} as the Bloch-Kato Lie algebras.
One of the aims of the present article is to explore this rather mysterious class of algebras. 

The main question that would establish a connection between the group and Lie algebra theoretical Bloch-Kato notions is the following.
\begin{questionA}\label{questA:BKgroupLie}
	 Is it true that the $p$-restricted Lie algebra associated to the Jennings-Zassenhaus filtration of a finitely generated Bloch-Kato pro-$p$ group is Bloch-Kato?
\end{questionA}
Since the cohomology of a Bloch-Kato Lie algebra is universally Koszul, a positive answer to Question A implies Weigel's strengthening of Positselski's Koszulity conjecture, as well as the universal Koszulity conjecture. 

Motivated by this connection, this paper explores the properties of quadratic, Koszul and Bloch-Kato Lie algebras over an arbitrary field $k$ of characteristic $\neq 2$. However, the analysis will not explicitly refer to the underlying group theoretic framework. 

In the first section, we introduce the key concepts of the Lie algebras under consideration and show that quadratic Lie algebras and quadratic  $p$-restricted Lie algebras can be treated simultaneously, provided that $p$ is an odd prime. 

The second section is devoted to the study of quadratic Lie algebras using HNN-extension. We establish that every quadratic Lie algebra decomposes as the HNN-extension over certain subalgebras generated in degree $1$. This result implies that if the quadraticity property is inherited by all subalgebras generated in degree $1$ of a quadratic Lie algebra $\li$, then $\li$ is Bloch-Kato, thereby simplifying the definition of this class of Lie algebras as it was introduced in \cite{sb}. 

HNN-extensions can also be used in the opposite direction for producing quadratic Lie algebras containing a given one. 
\begin{thmA}
	Any finitely presented graded Lie algebra embeds into a quadratic Lie algebra.
\end{thmA}
We conclude the section with two examples concerning the open problem of embedding finitely presented graded Lie algebras into Koszul ones.

The third section extends a result of \cite[Ex. 2, p. 22]{pp} on nilpotent Lie algebras satsfying the Koszul property, by showing that also solvable Koszul Lie algebras are necessarily abelian. This result could also been proved by using a combination of \cite{radical} and \cite[Ex. 2, p. 22]{pp}, while our proof does not require the theory of depth for Lie algebras, neither a growth argument. 

Moreover, we provide some constraints on the size and degrees of the center of a Koszul Lie algebra. 
\begin{thmB}
	The center of a Koszul Lie algebra of cohomological dimension $n$ is a finite dimensional ideal of dimension at most $n$ and it is concentrated in odd degrees $< \frac{n}2+1$. 
\end{thmB}
For proving the result, we make use of the theory of eigenvalues of algebras, i.e., the inverses of the roots of the Poincaré polynomial, developed by Weigel \cite{weig}. 

For Bloch-Kato Lie algebras, the situation is much simpler as the center is always concentrated in degree $1$, which allows one to decompose any such Lie algebra as a direct sum of its center and a centerless Bloch-Kato Lie algebra. 
As a byproduct of this result, we could prove that Bloch-Kato Lie algebras satisfy the Toral Rank Conjecture (TRC), which is believed to hold for nilpotent Lie algebras. 

\begin{corA}
	If $\li$ is a Bloch-Kato Lie algebra whose center has dimension $z$, then \[\dim H^\bu(\li,k)\geq 2^z.\]
\end{corA}

Following \cite{radical}, we study the radical of Koszul Lie algebras and prove that Bloch-Kato Lie algebras satisfy a decomposition of Levi-type.

\begin{thmC}
	If $\li$ is a Bloch-Kato Lie algebra with center $Z$, then $\li=\mc M\times Z$, where $\mc M$ is a Bloch-Kato Lie algebra that is essentially simple.
	\end{thmC}
	
	In the last section, we present two new class of Bloch-Kato Lie algebras of cohomological dimension $2$ and provide a characterization of the surface Lie algebras within the class of quadratic Lie algebras in terms of their standard subalgebras. 
	
	\begin{thmD}
		A non-free quadratic Lie $k$-algebra $\li$ is the graded Lie $k$-algebra associated to the lower central series of the fundamental group of an oriented closed surface iff any subalgebra of $\li$ generated by elements of degree $1$ is free.
	\end{thmD}

\section*{Acknowledgments}
	The author wishes to thank Conchita Mart\'inez-Pérez and Thomas Weigel for the suggestion of the topic, and Claudio Quadrelli for very helpful conversations. 
	
	This paper was also influenced by the discussions with Dietrich Burde, who encouraged the exploration of distinct features of finite dimensional algebras within this topic.
	
	The author is supported by the Austrian Science Foundation (FWF), Grant DOI 10.55776/P33811.

\section{Preliminaries}
A Lie algebra $\li$ over a field $k$ is called \textbf{graded} if it has a direct decomposition as a vector space, $\li = \bigoplus_{i=1}^\infty \li_i$ that is compatible with the Lie brackets. This means that for all $i, j \geq 1$, we have $[\li_i, \li_j] \subseteq \li_{i+j}$. We will tacitly suppose that each homogeneous component $\li_i$ has finite dimension over $k$. Such graded Lie algebras are called locally finite. 

If $V=\bigoplus_{i\in\mathbb Z}V_i$ is a graded vector space, then the free Lie algebra $\mc F(V)$ over $V$ inherits a natural grading. For a homogeneous subspace $R$ of $\mc F(V)$, we denote the quotient $\mc F(V)/(R)$ with the presentation $\pres{V}{R}$. If $(v_i)_{i\in I}$ and $(r_j)_{j\in J}$ are homogeneous bases for $V$ and $R$ respectively, we also write \[\pres{V}{R}=\pres{v_i:\ i\in I}{r_j:\ j\in J}.\]
For an integer $n$, we denote by $V[n]$ the vector space $V$ with grading $(V[n])_i=V_{n+i}$. The \textbf{Hilbert series} of $V$ is the formal series with non-negative coefficients \[H_V(t):=\sum_{i\in\mathbb Z}\dim (V_i)t^i\in \Z[[t,t^{-1}]]\]

\begin{exam}
	If $\Gamma=(\euV,\euE)$ is a finite combinatorial graph, i.e., $\euV$ is a finite set and $\euE$ is a set of $2$-element subsets of $\euV$, then the associated RAAG Lie $k$-algebra has presentation  
	\[\li_\Gamma=\pres{x_v: \ v\in \euV}{[x_v,x_w]:\ \{v,w\}\in \euE},\]
	where the vector space $V=\bigoplus_{v\in \euV}k\cdot x_v$ is concentrated in degree $1$, i.e., $V=V_1$. We will identify the vertices $v\in \euV$ with the corresponding generators $x_v\in\li_\Gamma$.
\end{exam}

A \textbf{graded connected} $k$-\textbf{algebra} is an associative $k$-algebra $A$ that admits a vector space decomposition $A = \bigoplus_{i=0}^\infty A_i$, where $A_0 = k \cdot 1_A$. The multiplication in $A$ must also respect this grading: for any $i, j \geq 0$, we have $A_i \cdot A_j \subseteq A_{i+j}$. An example of such an algebra is the universal enveloping algebra $\ul$ of a graded Lie algebra $\li$, which inherits a grading from that of $\li$.

A \textbf{graded $A$-module} is a left $A$-module $M$ that has a decomposition $M = \bigoplus_{i \in \mathbb{Z}} M_i$, consistent with the action of $A$, meaning that $A_i \cdot M_j \subseteq M_{i+j}$. 
For two graded left $A$-modules $M$ and $N$, and for any integer $j$, we denote by $\hom_A^j(M, N)$ the set of $A$-homomorphisms of degree $-j$. These are the $A$-linear maps $f: M \to N$ such that $f(M_i) \subseteq N_{i-j}$.

The functor $\hom^j_A(M,\argu)$ is left exact on graded $A$-modules and its right derived functor $R^\bullet \hom_A^j(M, -)$ is denoted by $\ext_A^{\bullet, j}(M, -)$. There are natural isomorphisms of the derived functors: \[(R^\bullet \hom_A^j(-, N))(M) \simeq (R^\bullet \hom_A^j(M, -))(N).\] 
One has a non-negative grading on these vector spaces
\[\ext_A^{\bu,j}(M,N)=\bigoplus_{i\geq 0}\ext^{i,j}_A(M,N),\] where $i$ is called the homological degree and $j$ the internal degree. 
We write $\ext^\bu_A(M,N)$ or $\ext^\bbu_A(M,N)$ for this bigraded vector space. 

When $M = k$ is the trivial $1$-dimensional $A$-module concentrated in degree $0$, we simplify the notation by setting $H^{\bbu}(A, N) = \ext_A^{\bbu}(k, N)$ and $H^{\bbu}(A) = \ext_A^{\bbu}(k, k)$. 
These are respectively the \textbf{cohomology spaces} of $A$ with coefficients in $N$, and the $k$-cohomology of $A$. 
 If $A = \ul$ for a graded Lie algebra $\li$, we write $H^{\bbu}(\li, N) = H^{\bbu}(\ul, N)$ and $H^{\bbu}(\li) = H^{\bbu}(\ul)$.

When $M = N$, the vector space $\ext_A^{\bbu}(M, M)$ can be equipped with a multiplication map, called the Yoneda product, turning it into a bigraded algebra. In particular, if $M = k$, then $H^\bbu(A, k)$ is a bigraded connected $k$-algebra. For the case $A = \ul$, the cohomology algebra $H^\bbu(\li)$ is graded commutative with respect to the homological degree, meaning that for any $a \in H^{i,r}(\li)$ and $b \in H^{j,s}(\li)$, we have $ab = (-1)^{ij}ba \in H^{i+j, r+s}(\li)$.

The cohomological dimension of a graded connected algebra $A$ is (see \cite{weig})\[\cd A=\sup\set{n\in \mathbb N}{H^n(A)\neq 0}=\min\left(\{\infty\}\cup\set{ n\in\mathbb N}{H^{n+1}(A)=0}\right)\]
If $n=\cd A$ is finite, then $\ext_A^m(M,N)=0$ for all $m>n$ and all graded $A$-modules $M$ and $N$. 
When $A$ is a projective module over a subalgebra $B$, then $\cd B\leq \cd A$. For instance, by the Poincaré-Birkhoff-Witt theorem, this happens when $A=\ul$ and $B=\mc U(\mc M)$, for a subalgebra $\mc M$ of a Lie algebra $\li$ over a field.

We say that $A$ is of type $FP_n$ if $H^i(A)$ has finite dimension for all $0\leq i\leq n$. If, moreover, $A$ has finite cohomological dimension $n$, then $A$ is said to be of \textbf{type FP}. The algebra $A$ is said to be \textbf{locally of type FP} if every finitely generated subalgebra of $A$ is of type FP.

\subsection{Quadratic and Koszul Lie algebras}\label{sub:quadkoszul}
Henceforth, we will focus exclusively on Lie algebras, which constitute the algebraic structures of interest in this paper. For a comprehensive study of Koszul algebras, we refer the reader to the book by Polishchuk and Positselski \cite{pp}.

 A graded Lie algebra $\li$ is called \textbf{standard} if $H^1(\li)$ has finite dimension and $H^{1,j}(\li)=0$ for $j\neq 1$. This amounts to saying that there exists a finite number of degree $1$ elements $x_1,\dots,x_n\in\li_1$ that generate $\li$ as a Lie algebra, or equivalently, $[\li,\li]=\bigoplus_{i\geq 2}\li_i$. Indeed, there is a graded isomorphism $H^{1,\bullet}(\li)\simeq\hom^\bu_k(\li/[\li,\li],k)$.
 
 If, moreover, $H^{2,j}(\li)=0$ unless $j=2$, then the Lie algebra is said to be \textbf{quadratic}. In other words, a quadratic Lie algebra is the quotient of a standard free Lie algebra $\mc F$ by an ideal generated by elements of $\mc F_2$. 
 The importance of quadratic Lie algebras comes from the fact that if $\li$ is any (locally finite) graded Lie algebra, then there is a unique quadratic Lie algebra $\q\li$, the \textbf{universal quadratic cover} of $\li$, such that \[\bigoplus_{i\geq 0}H^{i,i}(\li)\simeq \bigoplus_{i\geq 0} H^{i,i}(\q\li).\]
 Moreover, such a diagonal cohomology is a quadratic algebra (see \cite{lofwall},\cite{pp}). The Lie algebra $\q\li$ can be explicitly obtained from a minimal presentation of $\li$ in terms of generators and relations by neglecting the generators of degree $>1$ and the relations of degree $>2$.
 
A locally finite graded Lie algebra $\li$ is said to be \textbf{Koszul} when ${H^{i,j}(\li)}$ vanishes for all $i\neq j$. Such a Lie algebra is clearly quadratic, and its cohomology ring is a quadratic algebra, which can be computed explicitly (\hspace{1pt}\cite{lofwall}). Indeed, if $\li$ is any locally finite Lie algebra and $R$ is the kernel of the multiplication mapping $\li_1\wedge \li_1\to \ul_2$, the orthogonal complement $R^\perp$ of $R$ in $\li_1\wedge\li_1$ fits in the exact sequence \[0\to R^\perp\to \li_1^\ast\wedge\li_1^\ast\to R^\ast\to 0,\] and one has \[\li^!:=\bigoplus_{i\geq 0}H^{i,i}(\li)=\frac{\exa(\li_1^\ast)}{\exa(\li_1^\ast)\wedge R^\perp}\] where we have denoted by $\exa(V)$ the exterior algebra on a vector space $V$ concentrated in degree $1$. Clearly, Koszul Lie algebras are of type FP. 

If $n$ is a positive integer, then we say that a graded Lie algebra $\li$ is $n$-Koszul if $H^{i,j}(\li)=0$ for all $i\leq j\leq n$. In particular, $2$-Koszul Lie algebras are exactly the quadratic Lie algebras, and Koszul Lie algebras are $n$-Koszul for all $n$. These algebras are not to be confused with the $N$-Koszul algebras introduced by Berger \cite{N-koszul}, which are a non-quadratic generalization of Koszul algebras.

A standard subalgebra of a quadratic (or even Koszul) Lie algebra needs not to be quadratic itself (Examples \ref{ex:raag}, \ref{ex:2genKosz}). If this happens, then the Lie algebra is said to be \textbf{Bloch-Kato} (BK, for short).
\begin{defin}
	A graded Lie algebra $\li$ is said to be BK if it is hereditarily quadratic, i.e.,
	\begin{enumerate}
		\item $\li$ is quadratic, and 
		\item If $\mc M$ is a standard subalgebra of $\li$, then $\mc M$ is quadratic.
	\end{enumerate}
\end{defin} Notice that the definition is different -- and apparently weaker -- than that given in the author's paper \cite{sb} where the notion was first introduced. We will prove later (Corollary \ref{cor:bk}) that the two definitions are equivalent. 
It is worth noticing that the class of all BK Lie algebras is closed under taking both direct sums with standard abelian Lie algebras and free products of BK Lie algebras (\hspace{1sp}\cite{sb}).

\begin{exam}\label{ex:raag}
	If $\li_\Gamma$ is the RAAG Lie algebra associated to a graph $\Gamma$, then $\li_\Gamma$ is Koszul \cite{frob}. Moreover, $\li_\Gamma$ is BK if, and only if, $\Gamma$ is a Droms graph \cite{droms}, i.e., $\Gamma$ does not contain any square or line of length $3$ as an induced subgraph (see \cite[Ex. 4.4]{sb} and \cite{cassquad}). 
\end{exam}

If $\li$ is a quadratic Lie algebra, by a \textbf{quadratic filtration} we mean a series of quadratic subalgebras $\li=\li(0)\supseteq \li(1) \supseteq \li(2)\supseteq \dots\supseteq \li(n)\supseteq \li(n+1)=0$, where $\li(i+1)$ is a proper maximal standard subalgebra of $\li(i)$. In particular, $\dim\li(i)_1=\dim\li(i+1)_1+1$. 

In the following result we collect the equivalent definitions of a Koszul Lie algebra (cf. \cite{pp}).
\begin{thm}\label{thm:koszul}
	For a graded Lie algebra $\li$, the following statements are equivalent:
	\begin{enumerate}
		\item $\li$ is Koszul,
		\item The bigraded cohomology groups $H^{i,j}(\li)$ vanish unless $i=j$,
		\item $\li$ is standard and the cohomology ring $H^\bu(\li)$ is generated by $H^1(\li)$,
		\item The bigraded cohomology ring $H^\bbu(\li)$ is generated by $H^{1,1}(\li)$,
		\item The diagonal part $\bigoplus H^{i,i}(\li)$ of the cohomology ring of $\li$ is a Koszul algebra, and
		\item The $1$-dimensional trivial $\li$-module $k$ (concentrated in degree $0$) admits a free linear resolution, i.e., an exact sequence of $\ul$-modules $P_i$ \[ \xymatrix{\cdots\ar[r] & P_2 \ar[r]& P_1 \ar[r]& P_0 \ar[r] & k\ar[r]& 0}\] such that each $P_i$ is a graded free $\ul$-module generated by its elements of degree $1$, $P_i=\ul\cdot (P_i)_1$.
	\end{enumerate}
\end{thm}

\subsection{HNN-extensions}
Let $\mc M$ be a graded Lie algebra with a homogeneous subalgebra $\mc A$. If $\phi:\mc A\to \mc M$ is a derivation of degree $d$, one can form the HNN-extension of $\mc M$ with respect to $\phi$ and stable letter $t$. It is a Lie algebra $\hnn_\phi(\mc M,t)$ that is universal among the graded Lie algebras $\li$ with a homogeneous homomorphism $\psi:\mc M\to \li$ of degree $0$ such that there exists an element $t'\in\li$ satisfying $[t',\psi(a)]=\psi\phi(a)$.

Explicitly, if in the category of graded Lie algebras $\mc M$ has presentation $\pres{x_i:i\in I}{r_j:j\in J}$, then \[\hnn_\phi(\mc M,t)=\pres{t,\ x_i:\ i\in I}{[t,a]-\phi(a),\ r_j:\ a\in \mc A,\ j\in J}.\]
Notice that the HNN-extension as above inherits a grading from that of $\mc M$, where $t$ has degree $d$. 

As it was proved in \cite{hnnLS}, the Lie algebra $\mc M$ naturally embeds into $\li$. Moreover, by \cite{cmp}, there is a natural exact sequence of graded $\li$-modules \[0\longrightarrow \ul\otimes_{\mc U(\mc A)}k[-d]\longrightarrow \ul\otimes_{\mc U(\mc M)}k\longrightarrow k\longrightarrow 0\]
where all the maps have degree $0$. From the bigraded versions of Eckmann-Shapiro Lemma and of the cohomological long exact sequence (see \cite[Thms. 1.3, 1.4]{sb}), we deduce that there exist exact sequences of Mayer-Vietoris type for all $i,j\geq 0$ and all graded $\li$-modules $M$ \begin{equation}\label{eq:hnn les}H^{i-1,j-d}(\mc A,M)\to H^{i,j}(\li,M)\to H^{i,j}(\mc M,M)\to H^{i,j-d}(\mc A,M) \end{equation}

\subsection{Restricted Lie algebras} 
In case the characteristic $p$ of the ground field $k$ is positive, there exists the notion of a $p$-restricted (or $p$-) Lie algebra, that is an ordinary Lie algebra $\mf g$ endowed with a $p$-operation $(\argu)^{[p]}:\mf g\to \mf g$ satisfying several compatibility conditions with the sum and the Lie bracket of $\mf g$ (see \cite{restr}). 
If $\mf g=\li$ is graded and the map $(\argu)^{[p]}$ sends $\li_i$ to $\li_{ip}$, then $\li$ is said to be a \textbf{graded $p$-Lie algebra}. 

For this class of Lie algebras, the above notions of quadraticity, Koszul and Bloch-Kato properties can be easily defined analogously. For instance, the restricted universal enveloping algebra $\rul$ is defined as the quotient of the universal enveloping algebra by the ideal generated by all the homogeneous elements $x^{[p]}-x^p$ $(x\in \li)$, and the (bigraded) \textbf{restricted cohomology ring} of $\li$ is $H_r^\bbu(\li)=\ext_{\rul}^\bbu(k,k)$. 

However, if $p$ is odd, there is not much distinction between quadratic Lie algebras and $p$-restricted Lie algebras. Indeed, if $\li$ is a quadratic Lie algebra over a field of odd characteristic $p$, then the primitive elements of $\ul$ form a quadratic $p$-Lie algebra $\hat \li$ with restricted envelope $\ul$ (see \cite{milnorMoore}). 
In the respective categories of Lie algebras and of $p$-Lie algebras, $\li$ and $\hat\li$ share the same presentation. 
Conversely, given a quadratic $p$-Lie algebra $\hat\li$, one can find a presentation of $\hat\li$ in which no element of the form $x^{[p]}$ appears with non-zero coefficient in any relation, for, if $x$ has degree $1$, then $x^{[p]}$ has degree $p>2$. 
It follows that the standard non-restricted Lie subalgebra of $\hat\li$ generated by $\hat\li_1$ is quadratic. 
Again, $\ul\simeq \underline{u}(\hat\li)$. From this construction and the Poincaré-Birkhoff-Witt theorem it also follows that $\hat\li$ is torsion-free, i.e., the $p$-operation only vanishes on the zero element.

For all odd prime $p$ and field $k$ of characteristic $p$, we thus obtain a functor -- the $p$-\textbf{restrictification} -- that sends any quadratic Lie $k$-algebra $\li$ to a quadratic $p$-Lie $k$-algebra $\hat\li$ such that $\ul\simeq \underline{u}(\hat\li)$. 
By definition, the cohomology of $\li$ and the restricted cohomology of $\hat\li$ are naturally isomorphic.

All the results concerning quadratic $p$-Lie algebras can thus be easily obtained from the non-restricted case. For example, the Kurosh subalgebra theorem for BK Lie algebras \cite{sb} still applies to the restricted context.

\section{HNN-extensions and quadratic Lie algebras}
The usage of HNN-extensions is crucial for studying quadratic Lie algebras, due to the following fundamental, yet trivial, decomposition result. 
\begin{lem}\label{lem:quad hnn}
	Let $\li$ be a quadratic Lie algebra. Then, for every maximal proper standard subalgebra $\mc M$, there is a decomposition of $\li$ as the HNN-extension of $\mc M$ with respect to a derivation $\phi:\mc A\to \mc M$ of degree $1$ on a standard subalgebra $\mc A$ of $\mc M$. Explicitly, if $x$ is belongs to $\li_1$ but not to $\mc M$, then $\mc A$ can be chosen to be generated by the elements $m\in\mc M_1$ such that $[x,m]\in\mc M$.
	\begin{proof}
		Pick $x\in\li_1\setminus\mc M$. Fix a free Lie algebra $\mc F$ on the space $\li_1$ and identify its degree $1$ component with $\li_1$. Let $\mc G$ be the subalgebra of $\mc F$ generated by $\mi_1$. Let $r_1,\dots,r_m$ be minimal (quadratic) relations of $\li$, i.e., $\li=\mc F/(r_1,\dots,r_m)$. Since $\mc F_2=[\li_1,\li_1]=[\mc G_1,\mc G_1]\oplus[\mc G_1,x]$, we may assume that there is some $1\leq s\leq m$ such that $r_i$ belongs to $\mc G$ precisely when $1\leq i\leq s$. 
		Now, for every $s<i\leq m$, there are elements $a_i\in\mc M_1$ and $m_i\in\mc G_2$ such that $r_i=[x,a_i]+m_i$. 
		
		Let $\mc A$ be the subalgebra of $\mc M$ generated by the images of the elements $a_i$. Since $\mc A$ is a subalgebra of $\li$, the adjoint map $\ad(x):\li\to \li:y\mapsto [x,y]$ defines a derivation $\phi=\ad(x)\vert_{\mc A}:\mc A\to \mc M$. Notice that $\phi(a_i)=-m_i$.
		
		It is then clear that $\li$ is generated by $\mc M$ and $x$ with only additional relations given by $[x,a]=\phi(a)$, $a\in \mc A_1$.		
	\end{proof}
\end{lem}
Notice that the quadraticity of $\li$ is necessary for having a partition of the relation set that tells apart those relations in which $x$ appears in a single Lie bracket.

HNN-extensions can also be used for proving that a Lie algebra is Koszul. From the exact sequence (\ref{eq:hnn les}) it follows
\begin{lem}\label{lem:hnn_kosz}
	Let $\li=\hnn_\phi(\mc M,t)$, where $\mc A$ is a homogeneous subalgebra of a graded Lie algebra $\mc M$ and $\phi:\mc A\to \mc M$ is a derivation of degree $1$. Suppose that $\mc M$ is Koszul. Then, $\li$ is Koszul if, and only if, $\mc A$ is Koszul.
	
	Moreover, if this is the case, then the kernel of the restriction mapping $H^\bu(\li)\to H^\bu(\mc M)$ is the $H^\bu(\li)$-module $H^\bu(\mc A)[-1]$ and the Betti numbers $b_i(\argu):=\dim_k H^i(\argu)$ satisfy \[b_{i+1}(\li)=b_i(\mc A)+b_{i+1}(\mc M),\quad i\geq -1\]
\end{lem}

We have already noticed that Koszul Lie algebras are quadratic. By using the above Lemma, one can prove that also the converse holds, provided that all standard subalgebras are quadratic, i.e., when the Lie algebra is BK.
\begin{cor}\label{cor:bk}
	BK Lie algebras are Koszul. More precisely, all standard subalgebras of a BK Lie algebra are Koszul.
	\begin{proof}
		Let $\li$ be a BK Lie algebra. 
		
		We argue by induction on the minimal number $n$ of generators of $\li$, i.e., on the dimension of $\li_1$. 
		If $n=0$, there is nothing to prove, as $H^\bu(0)=H^{0,0}(0)=k$. 
		
		Let $n>0$ and suppose that all BK Lie algebras generated by at most $n-1$ elements are Koszul. 
		Consider an HNN-decomposition $\li=\hnn_\phi(\mc M,t)$, where $\phi:\mc A\to \mc M$ is a derivation of degree $1$ from a standard subalgebra $\mc A$ of $\mc M$ to the standard subalgebra $\mc M$ of $\li$. Clearly $\mc A$ and $\mc M$ are both BK; by induction, $\mc M$ and $\mc A$ are Koszul, and hence so is $\li$ by Lemma \ref{lem:hnn_kosz}.
	\end{proof}
\end{cor}
As mentioned above, this proves that the definition of a BK Lie algebra in \cite{sb} is equivalent to that we have given here. 

The proof of Corollary \ref{cor:bk} allows us to prove that RAAG Lie algebras are Koszul, without invoking Fr\"oberg's theorem \cite{frob}. 

\begin{prop}
	If $\Gamma$ is a finite simplicial graph, then its associated RAAG Lie algebra $\li_\Gamma$ is Koszul. 
	\begin{proof}
		We argue by induction on the number of vertices of $\Gamma$. If $\Gamma$ is the empty graph, there is nothing to prove. 
		
		Now suppose that $\Gamma$ contains a vertex $v$. Notice that if $\Delta$ is an induced subgraph of $\Gamma$, then the map $\li_\Delta\to\li_\Gamma$ induced by the inclusion $\Delta\subseteq \Gamma$ is well defined, and the identity of $\li_\Delta$ decomposes as $\li_\Delta\to\li_\Gamma\to\li_\Delta$, i.e., $\li_\Delta$ is a retract of $\li_\Gamma$. 
		
		If $\Gamma'$ is the maximal subgraph of $\Gamma$ that does not contain the vertex $v$, then, by induction, $\li_{\Gamma'}$ is Koszul. Moreover, \[\li_\Gamma\simeq \hnn_\phi(\li_{\Gamma'},v),\]
		where $\phi:\li_\Lambda\to \li_{\Gamma'}$ is the zero derivation and $\Lambda$ is the induced subgraph spanned by the vertices that are adjacent to $v$ in $\Gamma$. Since $\li_\Delta$ is Koszul by induction, we deduce that $\li_\Gamma$ is Koszul by Lemma \ref{lem:hnn_kosz}.
	\end{proof}
\end{prop}

  As done for Koszul Lie algebras in Theorem \ref{thm:koszul}, we collect the equivalent properties on a Lie algebra that make it BK (cf. \cite{sb}). 
  
  \begin{thm}
  	For a standard Lie algebra $\li$, the following statements are equivalent:
  	\begin{enumerate}
  		\item $\li$ is BK,
  		\item If $\mc M$ is a standard subalgebra of $\li$, then $\mc M$ is quadratic,
  		\item If $\mc M$ is a standard subalgebra of $\li$, then $\mc M$ is Koszul,
  		\item The cohomology ring $H^\bu(\li)$ of $\li$ is a universally Koszul algebra.
  	\end{enumerate}
  	  \end{thm}
  	  Recall that a graded connected algebra $A$ is said to be \textbf{universally Koszul} if \[\ext^{i,j}_A(I,k)=0\quad \text{for all }i\neq j-1\] and all the ideals $I$ of $A$ that are generated by elements of degree $1$.
  	  
  	  Let $\mc M$ be a Koszul subalgebra of a Koszul Lie algebra $\li$. Then, by the proof of \cite[Thm. A]{sb}, the cohomology of $\mc M$ is the quotient \[H^\bu(\mc M)\simeq H^\bu(\li)/(\mc M_1^\perp),\]
  	  of the cohomology of $\li$ with respect to the ideal generated by the set $\mc M_1^\perp$ of linear forms $\li_1\to k$ that vanish on $\mc M_1$. 
  	  
  	  Suppose now that $\li=\hnn_\phi(\mc M,t)$, where $\phi:\mc A\to \mc M$ is a derivation of degree $1$, and that $\mc A$ and $\mc M$, and hence $\li$, are all Koszul. 
  	  Since $\mc M_1$ has codimension $1$ in $\li_1$, the space $\mc M_1^\perp$ is linearly generated by an element $x\in H^1(\li)\simeq \li_1^\ast$.
  	  One has thus \begin{align*}
  	  	H^\bu(\mc M)\simeq H^\bu(\li)/(x),\\
  	  	H^\bu(\mc A)\simeq H^\bu(\li)/(0:x)
  	  \end{align*} 
  	  where $(0:x)$ is the annihilator of $x$ in $H^\bu(\li)$ that is generated by linear forms by \cite[Prop. 20]{MPPT}. 
  	  
\begin{prop}
	Let $\li$ be a BK Lie algebra with a homogeneous ideal $\mc M$. If $\mc M$ is a standard Lie algebra, then the quotient $\li/\mc M$ is BK.
	\begin{proof}
		Let $\mc N$ be a standard subalgebra of $\li$ containing $\mc M$. Since $\mc M$ is generated in degree $1$ and $\mc N$ is quadratic, the quotient $\mc N/\mc M$ is a quadratic Lie algebra. 
	\end{proof}
\end{prop}
  
  \subsection{Cocyclic ideals} 
If $\mc M$ is a homogeneous cocyclic ideal of a standard Lie algebra $\li$, i.e., $\li/\mc M$ is a $1$-dimensional Lie algebra, one easily sees that $\li$ can be decomposed into the semidirect product $\li=\mc M\rtimes k$, where $k\simeq \li/\mc M$ is the subalgebra of $\li$ generated by any $t\in\li_1\setminus \mc M$. For the purpose of this work, it is useful to notice that such a semidirect product is isomorphic to the HNN-extension $\hnn_\phi(\mc M,t)$, where $\phi:\mc M\to \mc M$ is the restriction of the adjoint map of $t$. 

By Lemma \ref{lem:hnn_kosz}, if $\mc M$ is Koszul, then so is $\li$. Under some finiteness assumptions on $\mc M$, also the converse holds.

\begin{cor}\label{cor:ideal FP}
	Let $\li$ be a Koszul Lie algebra and let $\mc M$ be a homogeneous cocyclic ideal of $\li$. Then, $\mc M$ is Koszul if, and only if, $\mc M$ is a Lie algebra of type FP.
	\begin{proof}
		As noticed above, since $\mc M$ is cocyclic, one has a decomposition of $\li$ into $\hnn_\phi(\mc M,t)$ for $\phi:x\in\mc M\mapsto [t,x]$. 
		One has the exact sequences\footnote{The exact sequence relating the cohomology of a Lie algebra with that of its ideals of codimension $1$ was discovered by Dixmier in \cite{dix}.} of equation (\ref{eq:hnn les}),   \[H^{ij}(\li)\to H^{ij}(\mc M)\to H^{i,j-1}(\mc M)\to H^{i+1,j}(\li),\] for every $j\geq i\geq 1$. 
		If $j>i+1$, as $\li$ is Koszul, the sequence gives rise to an isomorphism $H^{ij}(\mc M)\simeq H^{i,j-1}(\mc M)$.
		Now, if $\mc M$ was not Koszul, then there would exist some indices $q>p\geq 1$ such that $H^{pq}(\mc M)\neq 0$, and hence the above isomorphism would imply $H^{p,q+s}(\mc M)=H^{p,q}(\mc M)\neq 0$ for every $s\geq 0$, which contradicts the fact that $H^p(\mc M)$ is finite dimensional.
	\end{proof}
\end{cor}
In the same fashion, one can prove a similar result involving the weaker property of $n$-Koszulity.
\begin{prop}\label{prop:ideal FP}
	Let $\li$ be an $n$-Koszul Lie algebra, $n\geq 0$ and let $\mc M$ be a cocyclic ideal of $\li$. Then, for $0\leq m< n$, $\mc M$ is $m$-Koszul if, and only if, $\mc M$ is of type FP$_{m}$. 
\end{prop}

Recall that, for a Lie algebra $\li$ of type FP, the \textbf{Euler characteristic} of $\li$ is the integer \[\chi(\li):=P_\li(-1)=\sum_{i=0}^{\cd\li}(-1)^i\dim_k H^i(\li),\] where $P_\li\in \Z[t]$ denotes the \textbf{Poincaré series} of $\li$, i.e., the Hilbert series of $H^\bu(\li)$. If $\li$ is graded, then the non-negative integers $b_{ij}(\li)=\dim H^{ij}(\li)$ are called the \textbf{(bigraded) Betti numbers} of the Lie algebra. The $i$th Betti number of $\li$ is $b_i(\li)=\sum_j b_{ij}(\li)$.
\begin{cor}\label{cor:euler ch koszul}
	Let $\li$ be a Lie algebra of type FP with a cocyclic ideal $\mc M$ of type FP. Then, the Euler characteristic of $\li$ is zero. 
	\begin{proof}
		The long exact sequence
		 \ref{eq:hnn les}  \begin{align*}0\to H^0(\li)\to &H^0(\mc M)\to H^0(\mc M)\to H^1(\li)\to\dots\\&\dots\to H^d(\li)\to
		  H^d(\mc M)\to H^d(\mc M)\to 0\end{align*}  has finite length, where $d=\cd\li$. By hypothesis, all the spaces $H^i(\li)$ and $H^i(\mc M)$ are finite dimensional, and hence one has \[0=b_0(\mc M)-b_1(\li)+\dots+(-1)^d (b_d(\li)-b_d(\mc M)+b_d(\mc M))=\chi(\li).\]
	\end{proof}
\end{cor}
\begin{rem}
	Let $\li$ be Koszul and $\mc M$ be a graded cocyclic ideal of type FP, and hence Koszul.
	From Lemma \ref{lem:hnn_kosz}, we deduce that the cohomological dimension of $\mc M$ is $\cd\li-1$ and the Betti numbers satisfy \begin{equation}\label{eq:betti}b_{i+1}(\li)=b_{i+1}(\mc M)+b_{i}(\mc M).\end{equation}
	It follows that $b_n(\mc M)=\sum_{i=0}^n(-1)^{n-i}b_i(\li)$, for all $n\geq 0$.
	
	Eventually, for sufficiently large $n$,  \[0=b_n(\mc M)=\sum_{i=0}^n(-1)^{n-i}b_i(\li)=\pm\chi(\li).\]
	
	One may also observe that the equality (\ref{eq:betti}) implies $P_\li(t)=(1+t)P_{\mc M}(t)$, whence $\chi(\li)=P_\li(-1)=0$.
	One can now compute the Euler characteristic of $\mc M$ in terms of the Betti numbers of $\li$: \[\chi(\mc M)=\sum_{n=0}^d (-1)^{n+1} nb_n(\li)\] where $d=\cd\li$.
	
	Indeed, since $P_\li(t)=(1+t)P_{\mc M}(t)$, by taking derivatives, it follows that \[P'_\li(t)=(1+t)P'_{\mc M}(t)+P_{\mc M}(t)\] and hence $\chi(\mc M)=P_{\mc M}(-1)=P'_\li(-1)$.
\end{rem}

The Euler characteristic of a Koszul Lie algebra can be an arbitrary integer. In fact, for every integer $n$, there exists a Koszul Lie algebra $\li$ with $\chi(\li)=n$. For instance, if $n$ is a non-positive integer, the free Lie algebra $\mc F$ on $1-n$ elements satisfies $\chi(\li)=n$. 

If $n$ is positive, consider the graph $\Gamma_n$ with realization 
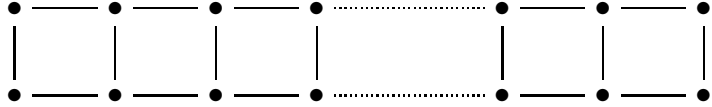
\begin{equation*}\label{eq:squaregraph intro}
	\xymatrix@R=1.5pt{ \bullet\ar@{-}[r] \ar@{-}[ddd]& \bullet\ar@{-}[r] \ar@{-}[ddd]& \bullet\ar@{-}[r] \ar@{-}[ddd] & 
		\bullet\ar@{.}[rr] \ar@{-}[ddd] && \bullet\ar@{-}[r] \ar@{-}[ddd] &  \bullet\ar@{-}[r] \ar@{-}[ddd] & \bullet\ar@{-}[ddd]  \\  \\  
		\\ \bullet\ar@{-}[r] &\bullet\ar@{-}[r] &\bullet\ar@{-}[r] & \bullet\ar@{.}[rr]  &&\bullet\ar@{-}[r] &\bullet\ar@{-}[r]& \bullet  }   \end{equation*}
	where the number of squares is $n$.
	Then, $\chi(\li_{\Gamma_n})=n$. The same can be done with $m$-gons, $m\geq 5$, instead of squares.
	
	However, for the RAAG Lie algebras associated to chordal graphs the situation is different.
\begin{prop}
	Let $\Gamma$ be a non-empty chordal graph. Then, $\chi(\li_\Gamma)\leq 0$. 
\end{prop}
\begin{proof}
	We argue by induction on the number of vertices of $\Gamma$.
	
	If $\Gamma$ is a vertex, then $\li_\Gamma$ is $1$-dimensional, and hence $\chi(\li_\Gamma)=0$. 
	
	Suppose that $\Gamma$ has more vertices. Recall that any chordal graph can be obtained by attaching two chordal graphs along a common subgraph which may be either complete or empty (\hspace{1pt}\cite{dromschordal}). If $\Gamma_1$ and $\Gamma_2$ are two such subgraphs with common complete subgraph $\Delta$, then there is a decomposition of $\li_\Gamma$ into the free product of $\li_{\Gamma_1}$ and $\li_{\Gamma_2}$ with amalgamated subalgebra $\li_\Delta$. 
	By \cite[Prop. 2.1]{cmp}, we deduce that \[\chi(\li_\Gamma)=\chi(\li_{\Gamma_1})+\chi(\li_{\Gamma_2})-\chi(\li_\Delta)\]
	
	Now, if $\Delta$ is the empty graph, then $\chi(\li_\Delta)=1$, else $\li_\Delta$ is non-zero and abelian and hence $\chi(\li_\Delta)=0$. In either cases, 
	\[\chi(\li_\Gamma)\leq \chi(\li_{\Gamma_1})+\chi(\li_{\Gamma_2})\]
	which is non-positive by induction.
\end{proof}
Since Droms graphs are chordal, we deduce that for RAAG Lie algebra that are BK, the Euler characteristic is non-positive. In light of the isomorphism given by the May spectral sequence \cite{may} for the cohomology of a  right-angled Artin (pro-$p$) group (see the Introduction), if $\Gamma$ is a chordal graph, then the associated (pro-$p$) group has non-positive Euler characteristic over any field $k$ (resp. over $\F_p$).
\begin{question}
	Can a BK Lie algebra have positive Euler characteristic?
\end{question}

\subsubsection{Fr\"oberg's formula}
For any Koszul (Lie) algebra $\li$, a distinguished formula holds,  \[\textbf{Fr\"oberg's formula}:\quad H_{\ul}(t)H_{\li^!}(-t)=1,\] where $\li^!$ is the diagonal cohomology of $\li$, which is in fact the whole cohomology .  In 1995, two works by Roos \cite{roos} and Positselski \cite{posit} proved that algebras satisfying Fr\"oberg's formula must not be Koszul. However, the following result shows that under suitable assumptions, Fr\"oberg's formula implies Koszulity of the Lie algebra. 
\begin{cor}
	Let $\li$ be a Koszul Lie algebra with a cocyclic ideal $\mc M$. Then, $\mc M$ is Koszul if, and only if, Fr\"oberg's formula holds for $\mc M$. 
	\begin{proof}
		From the Poincaré-Birkhoff-Witt Theorem it follows that the graded object $\gr\mc U(\mf g)$ associated with the canonical filtration of an ordinary Lie algebra $\mf g$ is the symmetric algebra on the vector space $\mf g$. Plus, if $\mf g=\li$ is $\mathbb N$-graded, then the Hilbert series of $\ul$ equals that of $\gr \mc U(\li)$ endowed with the grading induced by that of $\li$. In particular, \[H_{\ul}(t)=\prod_{i\geq 1}\frac1{(1-t^i)^{\ell_i}},\] where $\ell_i=\dim\li_i$ (see \cite[Ch. 2.2, Example 2]{pp}).
		
		Now, if $\mc M$ is a cocyclic ideal of $\li$, then $\dim\li_i=\dim\mc M_i+\delta_{1i}$, for all $i\geq 1$, where $\delta_{ij}$ is the Kronecker delta. Therefore, \[H_{\ul}(t)=\frac1{(1-t)}H_{\mc U(\mc M)}(t).\]
		
		The long exact sequence (\ref{eq:hnn les}) for $j=i\geq 0$ reads \begin{align*}H^{i-1,i}(\li)\to H^{i-1,i}(\mc M)\to H^{i-1,i-1}(\mc M)&\to H^{i,i}(\li)\to\\
			&\to H^{i,i}(\mc M)\to H^{i,i-1}(\mc M)=0.\end{align*}
			Since $\li$ is Koszul, and since each space $H^{i,j}(\mc M)$ is finite dimensional, we recover the following formulae involving the bigraded Betti numbers: 
		\[b_{i,i}(\li)=b_{i-1,i-1}(\mc M)+b_{i,i}(\mc M)-b_{i-1,i}(\mc M).\]
		In particular, the Hilbert polynomial of $\li^!$ (i.e., the Poincaré polynomial of $\li$) is given by 
		\begin{align*}H_{\li^!}(t)&=\sum_i b_{ii}(\li)t^i=\sum_i(b_{i-1,i-1}(\mc M)+b_{i,i}(\mc M)-b_{i-1,i}(\mc M))t^i=\\
			&=(1+t)H_{(\q\mc M)^!}(t)+tQ(t),
		\end{align*}
		where we have put $Q=\sum_i b_{i,i+1}(\mc M)t^i$. 
		
		Now, by Fr\"oberg's formula for $\li$, we get \begin{align*}1&=H_\ul(t) H_{\li^!}(-t)=\frac1{1-t}H_{\mc U(\mc M)}(t)\cdot \left((1-t)H_{(\q\mc M)^!}(-t)-tQ(-t)\right)=\\
			&=H_{\mc U(\mc M)}(t)H_{(\q\mc M)^!}(-t)-\frac t{1-t}H_{\mc U(\mc M)}(t)Q(-t).\end{align*}
		
		It follows that, if the Fr\"oberg formula holds for $\mc M$, then $Q(t)=0$, which in turn implies that $H^{i,i+1}(\mc M)=0$ for all $i\geq 0$. Eventually, from the proof of Corollary \ref{cor:ideal FP}, it follows that $H^\bbu(\mc M)$ is concentrated on the diagonal, proving that $\mc M$ is Koszul.
	\end{proof}
\end{cor}

  \subsection{Quadratic embeddings}
  
Lichtman and Shirvani \cite{hnnLS} show that any Lie algebra embeds into a simple one. Nevertheless, in the graded case one cannot expect to achieve a similar result, as all non-abelian graded Lie algebras have proper ideals (e.g., the commutator subalgebra). However, a similar usage of HNN-extensions allows us to deduce that, under mild assumptions, all graded Lie algebras embed into quadratic ones.

\begin{thm}\label{thm:embed quad}
	Every finitely presented graded Lie algebra can be embedded into some quadratic Lie algebra.
\end{thm}

First of all, we show that one can get rid of the high-degree generators of a finitely generated, non-standard Lie algebra, and embed it into a standard one. 

\begin{lem}\label{lem:embed std}
	Let $\li$ be a finitely generated graded Lie algebra. Then, $\li$ is a homogeneous subalgebra of a standard Lie algebra. 
	\begin{proof}
		Up to taking the direct product of $\li$ with any standard Lie algebra, we may suppose that $\li_1\neq 0$. 
		Let $\{x_i^n:\ (i,n)\in I\}$ be a minimal homogeneous generating system of $\li$, where $x_i^n\in\li_n$. We argue by induction on the maximal $N$ for which there is some generator $x_i^N$ of degree $N$.
		
		If $N=1$, then $\li$ is already standard. 
		
		Assume that $N>1$. Since $\li_1\neq 0$, pick $x\in\li_1\setminus\{0\}$ and consider, for all $(i,N)\in I$, the derivations $\phi_i:\gen x\to \li$ sending $x$ to $x_i^N$; such maps are homogeneous of degree $N-1$. Then, the iterated HNN-extension $\mc H$ of $\li$ with respect to all the derivations $\phi_i$, $(i,N)\in I$, is a Lie algebra containing $\li$ (by \cite{hnnLS}) and generated by the elements $x_i^n$, for $(i,n)\in I$ and $n<N$, and by the stable letters $t_i$ of degree $N-1$. By induction, $\mc H$ embeds into a standard Lie algebra, and hence so does $\li$.
	\end{proof}
\end{lem}
Notice that the standard Lie algebra $\mc S$ containing $\li$, as constructed in Lemma \ref{lem:embed std}, satisfies $\dim H^{2,j}(\mc S)= \dim H^{2,j}(\li)$ for all $j\geq 3$, provided that $\li_1\neq 0$. 
Indeed, if one puts $\li^{1}=\hnn_{\phi_1}(\li,t_1)$ and $\li^{i+1}=\hnn_{\phi_{i+1}}(\li^i,t_{i+1})$, then \[\mc S=\bigcup_{i:(i,N)\in I} \li^i\] and one has exact sequences {\[H^{1,j-N+1}(\gen{x})\to H^{2,j}(\li^{i+1})\to H^{2,j}(\li^i)\to H^{2,j-N+1}(\gen x)=0\] }for every $j$. If $j\geq N+1$, then $H^{1,j-1}(\gen x)=0$ and hence $H^{2,j}(\li^i)\simeq H^{2,j}(\li^{i+1})$. In particular, $\mc S$ has the same number of relations of all degrees $\geq N+1$ as $\li$, but more relations of degree $N$, where $N$ is the degree of the generator we want to get rid of. 
Eventually, if $\li$ is finitely presented, then the same holds for $\mc S$.

\begin{proof}[Proof of Theorem \ref{thm:embed quad}]
	By Lemma \ref{lem:embed std}, we may assume $\li$ to be a standard finitely presented Lie algebra. Let $d$ be the maximal degree of the minimal relations of $\li$ and assume that $d\geq 3$.
	
	Let $r=\sum_i [x_i,a_i]$ be a relation of degree $d$ of $\li$, where $(x_i)_{1\leq i\leq n}$ is a minimal generating system and the $a_i$'s are homogeneous elements of degree $d-1$. 
	
	Consider the direct sum $\mc Q^1=\li\times k$, where the abelian Lie algebra $k$ is generated by the degree-$1$ element $t_1$. 
	
	For $2\leq i\leq n$, we can define the derivations $\phi_i:t_1\mapsto a_i$ of degree $d-2$, and put \[\mc Q^2=\hnn_{\phi_i}(\mc Q^1,s_i)\] for the multiple HNN-extension of $\mc Q^1$ with respect to the derivations $\phi_i$, $2\leq i\leq n$. 
	
	Finally, let $\psi:\spn\{t_1,x_1\}\to\mc Q^2$ be the linear map defined by $t_1\mapsto a_1$ and $x_1\mapsto\sum_{i=2}^n[x_i,s_i]$. Such map is a derivation of the abelian Lie algebra $\gen{t_1,x_1}$ into $\mc Q^2$, as \begin{align*}
		[x_1,\psi(t_1)]+[\psi(x_1),t_1]&=[x_1,a_1]+\sum_{2\leq i\leq n}[[x_i,s_i],t_1]=\\
		&=[x_1,a_1]+\sum_{2\leq i\leq n}\left([[x_i,t_1],s_i]+[x_i,[s_i,t_1]]\right)=\\
		&=[x_1,a_1]+0+\sum_{2\leq i\leq n}[x_i,a_i]=\sum_{1\leq i\leq n}[x_i,a_i]=0.
	\end{align*}
	Hence, the HNN-extension $\mc Q=\hnn_\psi(\mc Q^2,t_2)$ has one relation in degree $d$ less than $\li$. Plus, $\mc Q$ has no relations of degree $> d$, but it is not $1$-generated as the elements $s_i$ are minimal generators of degree $d-2$. Notice that $\mc Q$ contains $\li$. By Lemma \ref{lem:embed std}, we get a standard Lie algebra $\li^{(r)}$ containing $\mc Q$, and hence $\li$, with no relations of degree $>d$ and with no more relations of degree $d$ than $\mc Q$. Indeed, as noticed above, $H^{2,j}(\li^{(r)})\simeq H^{2,j}(\mc Q)$ for every $j\geq d$.
	
	Since $\li$ is a standard finitely presented Lie algebra with relations of degree $\leq d$ only, one can thus proceed by induction. 
\end{proof}

The so obtained quadratic Lie algebra is far from being the minimal quadratic Lie algebra containing $\li$.

\begin{exam}\label{ex:heis}
	Consider the Heisenberg Lie algebra $\mf h_{n}$ of dimension $2n+1$: it can be given a basis $x_1,y_1,\dots,x_n,y_n,z$ with non-zero brackets $[x_i,y_i]=z$. If $n>1$, then $\mf h_{n}$ is quadratic, with (graded) presentation \[\pres{x_1,y_1,\dots,x_n,y_n}{[x_i,x_j],[y_i,y_j], \ [x_i,y_j]-\delta_{ij}[x_1,y_1]}.\]
	
	However, \(\mf{h}_n\) is not Koszul, as shown in \cite[Ex. 2, p. 22]{pp}. In Sections \ref{solvable} and \ref{center}, we will demonstrate this using two different approaches.
	
	Although \(\mf{h}_1\) is not quadratic, as it has minimal relations in degree 3, it can be embedded into the quadratic Lie algebra \(\mf{h}_n\) ($n \geq 2$).
\end{exam}

One may wonder if one can always embed finitely presented graded Lie algebras into Koszul ones, but this is not the case. 

\begin{exam}
		Consider the positive part $\mc W^+$ of the Witt Lie algebra. 
	
	One can give it the classical -- yet not minimal -- presentation 
	\[\mc W^+=\pres{x_i:\ i\geq 1}{[x_m,x_n]-(m-n)x_{m+n}}\]
	
	When the characteristic of the ground field is $0$, Goncharova \cite{goncharova} computed the $k$-cohomology of such a Lie algebra, discovering that it is $2$ dimensional in each positive homological degree, with a graded decomposition into nonzero components as follows:
	\[H^q(\mc W^+)=H^{q,q_1}(\mc W^+)\oplus H^{q,q_2}(\mc W^+)\]
	where $q_i=\frac{3q^2+(-1)^iq}{2}$.
	
	In particular, $\mc W^+$ is minimally generated by an element of degree $1$ and one of degree $2$. 
	The minimal relations lie in degree $5$ and $7$, giving the following minimal presentation \[\mc W^+=\pres{x_1,x_2}{r_5,r_7}\]
	
	where $x_i$ has degree $i$, and \begin{align*}
		r_5=&6[[x_2,x_1],x_2]-[[[x_2,x_1],x_1],x_1],\\
		r_7=&9[[[x_2,x_1],x_1],[x_2,x_1]]-[[[[x_2,x_1],x_1],x_1],x_1]
	\end{align*}

		The Lie algebra $\mc W^+$ is thus finitely presented but it has infinite cohomological dimension, implying that it does not embed into any Koszul Lie algebra. 
		
		Nevertheless, by Theorem \ref{thm:embed quad}, one can still embed it into a quadratic Lie algebra, which has thus infinite cohomological dimension.
		
		In order to make it clearer the construction of Lemma \ref{lem:embed std}, we build a standard Lie algebra containing $\mc W^+$. Let $\phi:\gen{x_1}\to \mc W^+$ be the derivation sending $x_1$ to $x_2$. 
		Then, $\li=\hnn_\phi(\mc W^+,t)$ is a standard Lie algebra generated by the two degree $1$ elements $t$ and $x_1$ with cohomology \[H^q(\li)=H^{q,q_1}(\mc W^+)\oplus H^{q,q_2}(\mc W^+),\quad q\geq 2.\]
\end{exam}
In particular, an obstruction for embedding a finitely presented Lie algebra into a Koszul one relies on the possibility of having infinite cohomological dimension. 
\begin{question}\label{quest:embed koszul}
	If $\li$ is a quadratic algebra of finite cohomological dimension, does there exist a Koszul Lie algebra containing it?
\end{question}

\begin{exam}\label{ex:2genKosz}
	The easiest example of a non-quadratic Lie algebra is \[\mf g=\pres{a,b}{[a,[a,b]]}\]
	Consider the Lie algebra \[\li=\pres{x,y,z,w}{[x,y]-[z,w],\ [x,w],\ [x,z]}\]
	
	It is a quadratic quotient of the surface Lie algebra (see Section \ref{sec:g2d}) \[\mc G_4=\pres{\bar x,\bar y,\bar z,\bar w}{[\bar x,\bar y]-[\bar z,\bar w]}.\] 
	
	Now, by \cite{sb}, the Lie subalgebra $\mc M$ of $\mc G_4$ generated by the elements $\bar y,\bar z$ and $\bar w$ is free, and hence the map $\phi:\mc M\to \mc M$ obtained by extending $\bar y\mapsto [\bar z,\bar w]$ and $\bar z,\bar w\mapsto 0$ is a derivation. 
	
	We can thus form the semidirect product $\mc M\rtimes_\phi \gen{t}$, which coincides with the HNN-extension \[\hnn_\phi(\mc M,t)=\pres{\bar y,\bar z,\bar w,t}{[t,\bar y]-[\bar z,\bar w],\ [t,\bar z],\ [t,\bar w]}\simeq \li\]
	
	In particular, $\mc M$ embeds into $\li$ as a cocyclic ideal, proving that $\li$ is Koszul of cohomological dimension $2$ by Lemma \ref{lem:hnn_kosz}. 
	
	Note that one has $[x,[x,y]]=0$ and hence the map $a\mapsto x$, $b\mapsto y$ extends to a Lie algebra homomorphism $\mf g\to \li$. In fact, it is not hard to see that it is also injective, proving that the non-quadratic Lie algebra $\mf g$ is contained in a Koszul one. In turn, $\li$ is not BK. 
	
	To see this, consider the maximal standard subalgebra $\mc B$ of $\li$ generated by $x,y$ and $w$. By using the exact sequence of Lemma \ref{lem:hnn_kosz}, we see that $\mc B$ has two minimal relations, i.e., $\mc B$ has minimal relations $[x,[x,y]]=0$ and $[x,w]=0$. Hence, the subalgebra generated by $x,y$ has a single relation, and it is thus isomorphic with $\mf g$. 
\end{exam}


The existence of a non-quadratic subalgebra generated by two elements in a Koszul Lie algebra (as in Example \ref{ex:2genKosz}) is a remarkable fact, as this does not hold in the large class of RAAG Lie algebras. Indeed, a Tits alternative type result holds for these Lie algebras:
\begin{prop}
	Let $\Gamma$ be a finite simplicial graph and let $x$ and $y$ be elements of degree $1$ in the RAAG Lie algebra $\li_\Gamma$. Then, the subalgebra they generate in $\li_\Gamma$ is either free or abelian.
\end{prop}
\begin{proof}
		If $\Gamma$ is a single vertex, there is nothing to prove.
	
	Assume $\Gamma$ has at least $2$ vertices. Denote by $\Gamma_x$ the induced subgraph of $\Gamma$ spanned by the vertices which appear with non-trivial coefficient in the expression of $x$ with respect to the canonical basis (i.e., the basis in bijection with the vertex set). Similarly define $\Gamma_y$. 
	
	If all vertices $v$ in $\Gamma_x$ and $w$ in $\Gamma_y\setminus\{x\}$ are adjacent in $\Gamma$, then $[x,y]=0$. 
	
	Assume that $v\in \Gamma_x$ and $w\in \Gamma_y\setminus\{x\}$ are not adjacent and let $\Delta$ be the induced subgraph of $\Gamma$ on the vertices $v,w$. Hence, $\li_\Delta$ is free.
	
		Consider the natural epimorphism $\pi:\li_\Gamma\to\li_\Delta$ that is the identity on the vertices of $\Delta$ and zero on the others. 
	Since $\pi\vert_{\gen{x,y}}$ is surjective and $\li_\Delta$ is free, it follows that $\gen{x,y}\simeq \li_\Delta$ is also free. 
\end{proof}
This phenomenon reflects that occurring for $2$-generator subgroups of right-angled Artin (pro-$p$) groups \cite{subraags} (resp. \cite{procraags}). 
\section{Solvability and center of Koszul Lie algebras}\label{solvable}
\subsection{Solvability}The aim of this section is to prove
\begin{thm}\label{thm:solvKosz}
	Let $\li$ be a Koszul Lie algebra. If $\li$ is solvable, then it is abelian. 
\end{thm}
Example \ref{ex:heis} shows that there exist solvable quadratic Lie algebras that are not abelian. In turn, Theorem \ref{thm:solvKosz} proves that those Lie algebras are not Koszul, independently from \cite{pp}.

It is known \cite[Ex. 2, p. 22]{pp} that a Koszul Lie algebra is either abelian or it has exponential growth, i.e., $\dim\li_n>c^n$, for some $c>1$. In particular, standard nilpotent Lie algebras are Koszul iff they are abelian. Here we provide another proof of the latter, by using Lemma \ref{lem:quad hnn}. Note that a finitely generated graded Lie algebra is nilpotent iff it has finite dimension.

\begin{prop}\label{prop:nilpKosz}
	If $\li$ is a nilpotent Koszul Lie algebra, then $\li$ is abelian. 
	\begin{proof}
		Let $\li$ be a minimal nilpotent non-abelian Lie algebra that is also Koszul.
		
		If $\mc M$ is a homogeneous cocyclic ideal of $\li$, then $\li$ admits a decomposition $\li=\mc M\rtimes k$. Now, $\mc M$ is nilpotent and hence of type FP. It follows that $\mc M$ is a cocyclic ideal of type FP of a Koszul Lie algebra, and hence Koszul as well. By minimality of $\li$, we deduce that $\mc M=\mc M_1$ is abelian. Now, if $x\in \mc M$, then for $t\in \li_1\setminus \mc M$, the element $[t,x]$ lies in $\mc M_2=0$, and hence $\li=\mc M\times k$ is abelian.
	\end{proof}
\end{prop}

For an ordinary Lie algebra $\mf g$, we denote by $\mf g^{(n)}$ the $n$th term of the derived series of $\mf g$, i.e., $\mf g^{(0)}=\mf g$ and $\mf g^{(n+1)}=[\mf g^{(n)},\mf g^{(n)}]$. Recall that $\mf g$ is $n$-step solvable if $\mf g^{(n)}=0$ and $\mf g^{(n-1)}\neq 0$. If $\mf g=\li$ is a graded Lie algebra, then each derived term $\li^{(n)}$ is a homogeneous subalgebra of $\li$.

\begin{lem}
	Let $\li$ be a graded locally-finite Lie algebra of finite cohomological dimension. If $\li^{(n)}$ has finite dimension for some $n\geq 1$, then so does $\li^{(n-1)}$.
	\begin{proof}
		We can assume that $\li^{(n)}\subseteq \bigoplus_{i=M}^N \li_{i}$ (e.g., one can set $M=2^n$, cf. \cite{jacob}).
		
		Consider now the Lie algebra \[\mc M=\li^{(n-1)}\cap \bigoplus _{i\geq 1}\li_{N+i}.\]
		
		By definition, $[\mc M,\mc M]\subseteq \li^{(n)}$, and $\mc M$ is a subalgebra of $\li$, hence $[\mc M,\mc M]\subseteq \li^{(n)}\cap \mc M=0$, i.e., $\mc M$ is an abelian Lie algebra. It follows that $\cd\li\geq \dim \mc M$, which implies that $\li^{(n-1)}$ must be finite dimensional as well, since the homogeneous components $\li_i$'s are all finite dimensional. 
	\end{proof}
\end{lem} 
With the following we conclude the proof of Theorem \ref{thm:solvKosz}
\begin{cor}\label{cor:solvFiniteCd}
	Let $\li$ be a graded, locally-finite, solvable Lie algebra of finite cohomological dimension. Then, $\li$ is finite dimensional (and hence nilpotent). 
	\begin{proof}
		Suppose $\li$ is $(n+1)$-step solvable, whence $\li^{(n)}$ is an abelian subalgebra of $\li$. Since $\dim\li^{(n)}=\cd\li^{(n)}\leq \cd\li$, we deduce that $\li^{(n)}$ has finite dimension, and so do all the $\li^{(r)}$'s for $r\leq n$. 
	\end{proof}
\end{cor}

Notice that in \cite{radical} a stronger version of Corollary \ref{cor:solvFiniteCd} was proved, as Lie algebras of finite cohomological dimension also have finite depth.

\subsection{The center of Koszul Lie algebras}\label{center}
Let $\li$ be a standard Lie algebra and assume that $z$ is a central element of $\li$.
If $z$ has degree $1$, then clearly $\li$ splits as the direct sum of standard algebras $\li/(z)$ and $\gen z$. Clearly, $\li$ is quadratic (resp. Koszul) precisely when $\li/ (z)$ is of the same type. 
Moreover, since abelian Lie algebras of dimension $d$ have cohomological dimension $d$, the dimension of the center cannot be greater than the cohomological dimension of the Lie algebra. 

\begin{rem}\label{rem:trc}
	If $\li$ is a graded Lie algebra whose center is concentrated in degree $1$, then the restriction map $H^\bu(\li)\to H^\bu(Z(\li))\simeq \exa(Z(\li))$ is surjective and hence \[2^{\dim Z(\li)}=\dim H^\bu(Z(\li))\leq \dim H^\bu(\li),\] i.e., the toral rank conjecture holds true for $\li$. 
	
	More generally, the same holds when $Z(\li)\cap[\li,\li]=0$.
\end{rem}

In the following, we will show that the center of Koszul Lie algebras must be concentrated in ``small" odd degrees. In turn, this shows that the quadratic Lie algebras of Example \ref{ex:heis} are not Koszul, independently from Theorem \ref{thm:solvKosz} and the growth argument used in \cite{pp}.
 
\begin{thm}\label{thm:centerKosz}
		Let $\li$ be a Koszul Lie algebra. Then, $Z(\li)$ is concentrated in odd degrees $<\cd\li/2+1$. 
\end{thm}

B\o gvad \cite{bogvad} proves the following result for the center of graded \textit{super}-Lie algebras of type FP.
\begin{lem}\label{lem:bogvad}
	Let $\li$ be a graded super-Lie algebra of type FP. 
	If $Z(\li)_n\neq 0$, then $1/H_\ul(t)$ is a polynomial divisible by $1-t^n$.
\end{lem}
This can also be applied to graded Lie algebras: If $\li$ is a graded Lie algebra, then it is also a super-Lie algebra $\mf g$ whose homogeneous components are concentrated in even degrees $\mf g_{2n}=\li_n$. Plus, $H_{\mc U(\mf g)}(t)=H_{\ul}(t^2)$. If $Z(\li)_n\neq 0$, then $Z(\mf g)_{2n}\neq 0$, and by Lemma \ref{lem:bogvad}, $1-t^{2n}$ divides $1/H_{\ul}(t^2)$, i.e., the polynomial $1/H_{\ul}(t)$ is divisible by $1-t^n$.
\begin{cor}\label{cor:oddDeg}
	Let $\li$ be a Koszul Lie algebra. Then, $Z(\li)$ is concentrated in degrees $<\cd\li/2+1$. 
	\begin{proof}
		Consider the universal envelope $\ul$. Since $\li$ is Koszul, it is of type FP and Fr\"oberg's formula gives $H_\ul(t)H_{H^\bu(\li)}(-t)=1$. Put $p(t)=H_{H^\bu(\li)}(t)$ and notice that $p(t)$ is a polynomial of degree $n=\cd\li$ with \textit{positive} coefficients. Explicitly, if $b_j=\dim H^j(\li)$ denotes the $j$th Betti number of $\li$, then \[p(t)=\sum_{j=0}^nb_jt^j.\]
		
		Let $z\in\li$ be a degree $i$ central element. By contradiction, assume $z\neq 0$, so that, by Lemma \ref{lem:bogvad}, $1-t^i$ divides $p(-t)$, i.e., there are integers $a_i$ such that \[p(-t)=(1-t^i)(a_0+a_1t+\dots+a_{n-i}t^{n-i}).\]
		By expanding the right-hand side \begin{equation}
			p(-t)=a_0+a_1t+\dots+a_{n-i}t^{n-i}-a_0t^i-a_1t^{i+1}+\dots-a_{n-i}t^n,
		\end{equation} we see that, if $n-i\leq i-2$, then the polynomial is written as a sum of monomials of increasing degree, and the coefficient of $t^{n-i+1}$ in $p(t)$ is zero. However, such a coefficient cannot vanish, since it equals $(-1)^{i+1}b_{i+1}$. 
	\end{proof}
\end{cor}
It follows from the above proof that the polynomial $p(t)$ with positive coefficients is divisible by $1-(-t)^i$, whenever $\li$ contains a non-trivial central element of degree $i$. Notice that one can recognise the parity of $i$ by looking at the roots of the polynomial $1-(-t)^i$: The integer $i$ is even iff $-1$ is a root of that polynomial. 

We are thus led to study the roots of the {Poincaré polynomial} of Koszul Lie algebras. This was one of the aims of Weigel's work \cite{weig}.

\subsubsection{Eigenvalues of a Lie algebra of type FP}

Let $\li$ be a Lie algebra of type FP.
Since the constant term of the Poincaré polynomial $P_\li(t)$ is non-zero, there are complex numbers $\lambda_1,\dots,\lambda_n$ such that \[P_\li(t)=\prod_{i=1}^n(1+\lambda_it).\] These complex numbers are called the \textbf{eigenvalues} of $\li$. 

As noticed by Weigel in \cite{weig}, there is a constraint on the real eigenvalues, which relies on Descartes criterion on the sign of real roots of a real polynomial.

\begin{fact}\label{fact:positive eig}
	The real eigenvalues of a Lie algebra of type FP are positive.
\end{fact}

From this, we derive:
\begin{cor}
	Let $\li$ be a Koszul Lie algebra. Then, the center of $\li$ is concentrated in odd degrees. 
	\begin{proof}
		Let $z$ be a non-trivial central element of $\li$ of degree $i$. Hence, as above, $1-(-t)^i$ divides the Poincaré polynomial $P_\li(t)$. Now, by Fact \ref{fact:positive eig}, all real eigenvalues of $\li$ are positive. It follows that $1-(-t)^i$ has no negative root, and hence $i$ must be odd.
	\end{proof}
\end{cor}
Together with Corollary \ref{cor:oddDeg}, one deduces Theorem \ref{thm:centerKosz}.

\subsection{The $b_2$-conjecture}
Notice that the definition of eigenvalues applies to any associative algebra of type FP. We can give a partial answer to Question 2 of \cite{weig} (see also \cite[Question 3]{weig2}).
\begin{prop}\label{prop:gocha}
	Let $A$ be an algebra of type FP$_\infty$ and finite cohomological dimension $n\geq 1$. If all the eigenvalues of $A$ are real, then \[b_2(A)\leq \frac{n-1}{2n}b_1(A)^2.\]
	More generally, for $j=0,\dots,n-1$,
	\[b_{j-1}(A)b_{j+1}(A)\leq \frac{j(n-j)}{(j+1)(n-j+1)}b_j(A)^2\]
	\begin{proof}
		If $\lambda_1,\dots,\lambda_n\in \mathbb R$ are the eigenvalues of $A$, then they are positive by Fact \ref{fact:positive eig} and \begin{align*}
			b_1(A)&=\sum_{i=1}^n \lambda_i,\\
			b_2(A)&=\sum_{1\leq i< j\leq n}\lambda_i\lambda_j.
		\end{align*}
		In particular, $b_1(A)$ and $b_2(A)$ are elementary symmetric polynomials. The result thus follows from the well known Newton's inequalities.
	\end{proof}
\end{prop}
Weigel \cite[Question 2]{weig} asks whether the above inequality involving $b_1$ and $b_2$ holds for any Koszul algebra $A$ of finite cohomological dimension. 

\begin{rem}
	If $n=2$, then, both the eigenvalues of $A$ are real numbers (see \cite{weig}), and hence $b_2(A)\leq b_1(A)^2/4$. 
	
	The celebrated Golod-Shafarevich theorem states that a pro-$p$ group $G$ satisfying $b_2(G,\F_p)\leq b_1(G,\F_p)^2/4$ is infinite. For finite-dimensional nilpotent Lie algebras, the last inequality has been conjectured to hold in the opposite direction (\hspace{1sp}\cite{b2conj}).
\end{rem}

As noticed in \cite{weig}, there exist Koszul Lie algebras with some non-real eigenvalues. We now provide an example of a Koszul Lie algebra having complex eigenvalues with negative real parts,  answering negatively to \cite[Question 1]{weig}. 

\begin{exam}
	Let $\Gamma$ be the graph obtained as the disjoint union of a complete graph on $7$ vertices and $8$ isolated vertices. Then, the RAAG Lie algebra $\li_\Gamma$ is Koszul (see \cite{cmp}) with Poincaré polynomial (that equals the clique polynomial of $\Gamma$) \[P_{\li_\Gamma}(t)=(1+t)^7+8t.\]
	A numerical computation shows that some of the roots of $P_{\li_\Gamma}(-t)$ are approximately $1/\lambda_\pm\approx-0.02463\pm0.80986\ i$, and hence the eigenvalues $\lambda_\pm$ have negative real parts.
	
	In fact, the set of complex roots of all clique polynomials is dense in $\mathbb C$, as it was proved in \cite{rootsdense}.
\end{exam}
By Tur\'an's Theorem \cite{turan}, if a graph $\Gamma$ on $n$ vertices does not contain any $(r+1)$-clique, then, the number of edges of $\Gamma$ does not exceed \[{\mathbf e}_{\tiny{\txt{max}}}=\left(1-\frac1r\right)\frac{n^2-s^2}2+{s\choose2},\] where $0\leq s<r$ and $s\equiv n\mod r$. The number $\displaystyle \mathbf e_{\tiny\txt{max}}$ is the number of edges of the Tur\'an graph $T(n,r)$.
In particular, as noticed in \cite{weig}, if $r\geq 2$, then $\mathbf e_{\tiny\txt{max}}\leq \frac{r-1}{2r}n^2$, proving that, for all RAAG Lie algebras, the formula of Proposition \ref{prop:gocha} holds, despite their eigenvalues might not be real. 
In particular, if $\Gamma$ is a graph with $v$ vertices and $e$ edges, then \[n\geq \frac{v^2}{v^2-2e}\]that gives the same lower bound for the clique number of $\Gamma$ as that appearing in \cite{myersLiu}.

Given the opposite nature of nilpotent Lie algebras and (non-abelian) BK Lie algebras, we suspect that the latter class satisfies a version of the $b_2$-conjecture in the opposite direction, i.e.,
\begin{conj}[BK version of the $b_2$-conjecture]\label{conj:b2}
	If $\li$ is a BK Lie algebra of cohomological dimension $n$, then \[\omega(\li):=(n-1)b_1(\li)^2-2nb_2(\li)\geq 0.\]
\end{conj}
This is a special case of Question 2 of \cite{weig}, and it is known to hold when $n\leq 2$.

In order to attack the conjecture, one might try to use induction, due to the hereditary property of BK Lie algebras. 
\begin{lem}\label{lem:b2}
	Let $\li$ be a quadratic Lie algebra of cohomological dimension $n\geq 2$ with standard subalgebras $\mc M$ and $\mc A$ such that $\li=\hnn_\phi(\mc M,t)$, and $\phi:\mc A\to \mc M$ is a derivation of degree $1$. 
	
	Suppose that $\cd\mc M=n-1$ and that \[\omega(\mc M):=(n-2)m_1^2-2(n-1)m_2\geq 0,\] where $m_i=b_i(\mc M)$. 
	Then, $\omega(\li)\geq 0$.
\end{lem}
\begin{proof}
	Since $\dim\mc A_1\leq m_1$, we have
	\begin{align*}
		\omega&(\li)=(n-1)(m_1+1)^2-2n(m_2+\dim\mc A_1)=\\
		&\geq\omega(\mc M)+m_1^2+(n-1)(2m_1+1)-2m_2-2nm_1\geq\\
		&\geq \omega(\mc M)+m_1^2-2m_1+n-1-2m_2=\\
		&=\omega(\mc M)+m_1^2-2m_1+n-1+\frac1{n-1}(\omega(\mc M)-(n-2)m_1^2)=\\
		&=\frac{n}{n-1}\omega(\mc M)+\frac1{n-1}(m_1-n+1)^2\geq 0\end{align*}
		
	i.e., $\li$ satisfies the BK version of the $b_2$-conjecture.
\end{proof}
In particular, if all the BK Lie algebras of cohomological dimension $n$ have a maximal standard subalgebra of cohomological dimension $n-1$, then Conjecture \ref{conj:b2} holds true (see Corollary \ref{cor:b2}).

\begin{lem}\label{lem:homRAAG}
	Let $\Gamma$ be a Droms graph of clique number $n>1$. Then, there exists a maximal standard subalgebra $\mc M$ of the RAAG Lie algebra $\li_\Gamma$ of cohomological dimension $n-1$. 
	\begin{proof}
		We argue by induction on the number of vertices of $\Gamma$. If $\Gamma$ consists of a single edge, and hence $n=2$, then the Lie algebra generated by anyone of the vertices of $\Gamma$ has the expected cohomological dimension. 
		
		Assume now that $\Gamma$ has more than one edge. 
		
		If $\Gamma$ is connected, then, by the main Lemma of \cite{droms}, there is an induced subgraph $\Gamma'$ and a vertex $v$ not in $\Gamma'$ such that $\Gamma$ is the cone on $\Gamma'$ with $v$ as a tip, i.e., $v$ is adjacent to all other vertices of $\Gamma$. 
		It follows that $\li_\Gamma\simeq \li_{\Gamma'}\times k$ and hence $\mc M=\li_{\Gamma'}$ has the right cohomological dimension. 
		
		If $\Gamma$ is not connected, then there are two proper induced subgraphs $\Gamma_i$, $i=1,2$, such that $\Gamma$ is their disjoint union. If we denote by $\li(i)$ the RAAG Lie algebra on $\Gamma_i$, we get a free product decomposition $\li_\Gamma=\li(1)\amalg\li(2)$. 
		If the clique numbers of $\Gamma_1$ and $\Gamma_2$ are equal (and hence $>1$), by induction, for $i=1,2$, $\li(i)$ contains a maximal standard subalgebra $\mc M(i)$ of cohomological dimension one less than that of $\li(i)$, i.e., $\cd\mc M(i)=n-1$. 
		Note that $\li_\Gamma$ is BK, by \cite{sb}, and hence its standard subalgebras satisfy a version of the Kurosh subalgebra theorem (\hspace{1sp}\cite{sb}).
		If $v_i$ is a vertex of $\Gamma_i$ such that its associated standard basis element does not belong to $\mc M(i)$, then, by the Kurosh theorem, the Lie algebra $\mc M$ generated by $\mc M(1)$, $\mc M(2)$ and $v_1+v_2$ decomposes as the free product \[\mc M=\mc M(1)\amalg\mc M(2)\amalg \gen{v_1+v_2}\]
		In particular, $\cd\mc M=\max(\cd\mc M(1),\cd\mc M(2),1)=n-1$. 
		
		It remains to consider the case when $\Gamma_1$ and $\Gamma_2$ have different clique numbers. In that case, assume the clique number of $\Gamma_1$ is $n$ and that of $\Gamma_2$ is $<n$. By defining $\mc M(1)$ as above, then the subalgebra generated by $\mc M(1)$ and $\li(2)$ equals their free product and hence it is a maximal standard subalgebra of cohomological dimension $n-1$.
	\end{proof}
\end{lem}


For RAAG Lie algebras we have a Lie theoretic counterpart of the main theorem of Droms' \cite{droms} (and of its pro-$p$ version \cite[Theorem 1.2(ii)]{ilirPavel}). As proved in \cite{cmp}, the RAAG Lie algebra $\li=\pres{a,b,x}{[a,b]}$ contains a non-standard subalgebra $\mc M=\gen{a,b,[x,a],[x,b]}$ that is not isomorphic to any RAAG Lie algebra. However, $\mc M$ admits a standard grading which makes it into a quadratic Lie algebra
\[\mc M=\pres{a,b,z,t}{[a,b],[z,a]+[t,b]}.\]
By Theorem \ref{thm:2rel}, $\mc M$ is BK with respect to the new grading.
\begin{prop}\label{prop:raagET}
	If $\Gamma$ is a finite simplicial graph, then the RAAG Lie algebra $\li_\Gamma$ is BK iff all of its standard subalgebras are isomorphic with a RAAG Lie algebra.
	\begin{proof}
		We translate almost verbatim the pro-$p$ group theoretic proof of \cite{ilirPavel} to our scope.
		
		If all standard subalgebras of $\li_\Gamma$ are RAAG Lie algebras, then $\li_\Gamma$ is clearly BK.
		
		Assume now that $\li_\Gamma$ is BK.
		We prove the result by induction on the number of vertices of $\Gamma$. If $\Gamma$ consists of a single vertex, the result is clear. Suppose that $\Gamma$ has more than one vertex. Decompose $\Gamma$ into its connected components $\Gamma_i$, $1\leq i\leq k$. 
		
		If $k\geq 2$, then every component $\Gamma_i$ is a Droms graph, and hence, by induction, every subalgebra of $\li_{\Gamma_i}$ is a RAAG Lie algebra.
		Let $\mc M$ be a $1$-generated subalgebra of $\li_\Gamma$; by the Kurosh subalgebra theorem \cite{sb}, $\mc M$ is a free product of a free Lie algebra and subalgebras of the $\li_{\Gamma_i}$'s, and thus it is a RAAG Lie algebra. 
		Explicitly, it is the algebra on the graph obtained from the disjoint union of the graphs corresponding to the intersections $\mc M\cap \li_{\Gamma_i}$, and a finite number of isolated vertices.
		
		Now suppose that $\Gamma$ is connected, and thus it is a cone $\nabla(\Gamma')=\Gamma$, with tip $v$.
		Then $\li_\Gamma=\gen v\times \li_{\Gamma'}$. 
		Let $\phi:\li_{\Gamma}\to\li_{\Gamma'}$ be the natural projection. 
		Let $\mc M$ be a $1$-generated subalgebra of $\li_\Gamma$. Then we have the following central extension \[0\to \mc M\cap \gen v\to \mc M\to \phi(\mc M)\to 0.\]
		
		We claim that this sequence splits. 
		Since $\Gamma'$ is a Droms graph, by induction $\phi(\mc M)$ is a RAAG Lie algebra, say $\phi(\mc M)=\li_{\Delta}$.
		Let $u$ be a vertex of $\Delta$, and choose $m\in \mc M$ such that $\phi(m)=u$. 
		As $u\in \li_{\Gamma'}$, one has $\phi(u)=u$, and hence, $u-m\in \ker\phi\leq\gen v$.
		This means that there is a scalar $\alpha_u\in k$ such that $u+\alpha_uv=m\in \mc M$.
		Define $\rho_1:\phi(\mc M)_1\to \mc M_1$ by linearly extending $\rho_1(u)=u+\alpha_u v$, $\forall u\in V(\Delta)$. 
		For $x=\sum_{u\in V(\Delta)}r_u u\in \phi(\mc M)_1=\text{Span}_k V(\Delta)$, put $ \alpha_x =\sum_{u\in V(\Delta)}r_u \alpha_u$, so that \[\rho_1(x)=x+\alpha_x v,\ \forall x\in \phi(\mc M).\]
		
		Since $v$ is in the center of $\li_\Gamma$, for $\{x,x'\}\in E(\Gamma')$, we have \[[\rho_1(x),\rho_1(x')]=[x+\alpha_xv,x'+\alpha_{x'}v]=[x,x']=0,\] whence $\rho_1$ extends to a well-defined Lie algebra homomorphism $\rho:\phi(\mc M)\to \mc M$ 
		Moreover, $\phi\rho(u)=\phi(u+\alpha_u v)=\phi(u)=u$, for every $u\in V(\Delta)$, and thus $\phi\rho=\text{Id}_{\phi(\mc M)}$, i.e., $\rho$ is a section.
		Now, since $\mc M\cap\gen v$ is contained in the center of $\mc M$, we have $\mc M=(\mc M\cap \gen v)\times \phi(\mc M)$. 
		If $\mc M\cap \gen v=0$, then $\mc M\leq\li_{\Delta}$ is a RAAG Lie algebra by induction; otherwise, $\mc M$ contains $v$, and hence we have $\mc M\leq\li_{\nabla(\Delta)}$.
	\end{proof}
\end{prop}

From Lemma \ref{lem:homRAAG} and Proposition \ref{prop:raagET} it follows a proof of the BK version of the $b_2$-conjecture for the RAAG Lie algebras on Droms graphs, independently from T\'uran's.

\subsubsection{The free rank}

For a standard Lie algebra $\li$, define the \textbf{free rank} of $\li$ as the maximal dimension of the generating spaces of standard free subalgebras of $\li$, i.e., \[\frk \li:=\max\set{\dim V}{V\subseteq \li_1,\ \gen V \text{ is free}}\]

\begin{lem}Let $\li$ be a BK Lie algebra. Then, $\frk\li\leq \dim\li_1-\cd\li+1$.
	\begin{proof}
		Let $\mc F$ be a standard free Lie subalgebra of $\li$ on $r$ elements. Set $d=\dim\li_1$ and $n=\cd\li$. 
		
		Let $\mc M^{(1)}$ be a standard subalgebra of $\li$ containing $\mc M^{(0)}=\mc F$ with $\dim\mc M^{(1)}_1=r+1$. Similarly define $\mc M^{(i)}$ for all $1\leq i\leq d-r$, so that $\mc M^{(d-r)}=\li$. 
		
		Since $\mc M^{(i)}$ is a maximal standard subalgebra of $\mc M^{(i+1)}$ ($0\leq i\leq d-r-1$), one has \[\cd\mc M^{(i+1)}\leq \cd\mc M^{(i)}+1,\] and hence \begin{align*}
			n&=\cd\mc M^{(d-r)}=\cd\mc M^{(d-r)}-\cd\mc M^{(d-r-1)}+\cd\mc M^{(d-r-1)}=\\
			&=\sum_{i=1}^{d-r}\left(\cd\mc M^{(d-r-i+1)}-\cd\mc M^{(d-r-i)}\right)+\cd\mc M^{(0)}\leq \\
			&\leq \left(\sum_{i=1}^{d-r}1\right)+\cd\mc F=d-r+1.
		\end{align*}
		which proves the claim.
	\end{proof}
\end{lem}
It follows from the proof that $\frk\li=\dim\li_1-\cd\li+1$ is equivalent to requiring that the cohomological dimensions of the $\mc M^{i}$ is strictly decreasing.

One can state a dual of the latter equality. Let $A$ be a graded-commutative universally Koszul algebra. Its cohomology ring is the universal enveloping algebra of a BK Lie algebra $\li$ (\hspace{1sp}\cite{sb}).
The equality $\frk\li=\dim\li_1-\cd\li+1$ is equivalent to $A$ containing elements $x_1,\dots,x_n\in A_1$ such that $x_1A_{n-s}+\dots+x_sA_{n-s}=A_{n-s+1}$, for all $1\leq s\leq n $, where $n=\cd\li$. 
\begin{cor}\label{cor:b2}
	If every BK Lie algebra $\li$ satisfies $\frk \li= \dim\li_1-\cd\li+1$, then the BK version of the $b_2$-conjecture holds true.
	\begin{proof}
		One can argue by induction on the minimal number of generators of a BK Lie algebra and apply Lemma \ref{lem:b2}.
	\end{proof}
\end{cor}

\begin{question}\label{quest:frk}
	Is it true that 
	\[\frk\li=\dim\li_1-\cd\li+1\]
	for any BK Lie algebra $\li$? In other words, do all BK Lie algebras of cohomological dimension $n$ have maximal proper standard subalgebras of cohomological dimension $n-1$?
	
	Dually, is it true that if $A$ is a universally Koszul graded commutative algebra of maximal degree $n$, then there exists a degree $1$ element in $A$ such that $A_{n-1}x=A_n$?
\end{question}
Since any RAAG Lie algebra of cohomological dimension $n$ has a maximal standard subalgebra of dimension $n-1$ by Lemma \ref{lem:homRAAG}, Question \ref{quest:frk} has positive answer for this class.

\subsection{The center of Bloch-Kato Lie algebras}
For BK Lie algebras, determining the center is much simpler. 
\begin{thm}\label{thm:centerBK}
	Let $\li$ be a BK Lie algebra. Then, the center of $\li$ is concentrated in degree $1$. 
	\begin{proof}
		We argue by induction on $\dim\li_1$. 
		Denote by $Z=Z(\li)$ the center of $\li$.
		
		If $\dim\li_1=1$, then $\li=Z(\li)$ is abelian, concentrated in degree $1$.
		
		Let $\dim\li_1>1$ and suppose that if $\mc M$ is any BK Lie algebra with $\dim\mc M_1<\dim\li_1$, then the center of $\mc M$ is concentrated in degree $1$. 
		
		Decompose $\li$ into the HNN-extension $\li=\hnn_\phi(\mc M,\mc A,t)$, where $\mc M$ is a standard subalgebra of $\li$, $\mc A$ is a standard subalgebra of $\mc M$ and $\phi:\mc A\to \mc M$ is a derivation of degree $1$. In particular, $\mc M$ is BK, and $\dim\mc M_1=\dim\li_1-1$. Hence, by induction, one has $Z_n\cap \mc M\subseteq Z(\mc M)_n=0$, for $n\geq 2$.
		
		Fix an element $z\in Z$ of degree $n\geq 2$. 

		If $\mc M=\mc A$, then $\li_n=\mc M_n$, for $n\geq 2$, proving that $z\in Z(\mc M)$, and hence $z=0$. 
		
		We may thus assume that $z\notin \mc M$ and $\mc M\neq \mc A$.
		Consider the two-dimensional (abelian) Lie algebra $B$ generated by $z$ and an element $m\in \mc M_1\setminus \mc A$. By \cite[Thm. 3]{hnnLS}, $B$ is a subalgebra of the free product $C\amalg B_0$, where $C$ is a free subalgebra of $\li$ and $B_0=\mc M\cap B=\gen{m}$. Since the free Lie algebra of rank $\geq 2$ has no abelian subalgebra of dimension $\geq 2$, we get a contradiction, whence $z$ must be zero.
	\end{proof}
\end{thm}
Note that for the aim of the result, $\mc A$ might have center in arbitrary degree. The same proof can thus be adapted to Lie algebras having quadratic filtrations by taking $\mc M=\mc M(1)$ (cf. \ref{sub:quadkoszul}) and then arguing by induction.

In light of Remark \ref{rem:trc}, we get the Toral Rank Conjecture for BK Lie algebras.
\begin{cor}
	If $\li$ is BK (or has a quadratic filtration), then \[\dim H^\bu(\li)\geq 2^{\dim Z(\li)}.\]
\end{cor} 

In general, the previous result cannot be extended to Koszul Lie algebras with the same proof, since the center of such algebras is not known to lie in degree $1$ in general. Nevertheless, for RAAG Lie algebras this holds. 

\begin{prop}
	If $\Gamma=(V,E)$ is a finite simplicial graph, then the associated RAAG Lie algebra \[\li_\Gamma=\pres{v\in V}{[v,w]:\ v,w\in E}\]
	has center concentrated in degree $1$ and satisfies the TRC. 
	\begin{proof}
		If $\Delta$ is an induced subgraph of $\Gamma$, then the associated RAAG Lie algebra $\li_\Delta$ of $\Delta$ is naturally a subalgebra of $\Gamma$. It follows that $\li_\Gamma$ admits a quadratic filtration, and hence its center is concentrated in degree $1$ for the proof of Theorem \ref{thm:centerBK}.
	\end{proof}
\end{prop}

\begin{question}\label{quest:koszulquadfilt}
	Do Koszul Lie algebras have quadratic filtrations? 
\end{question}
Affirmative answer to Question \ref{quest:koszulquadfilt} would imply that the center of any Koszul Lie algebra is concentrated in degree $1$ (making the arguments of Theorem \ref{thm:centerKosz} unnecessary), and hence that TRC holds true for such Lie algebras. 

\subsection{Essential decompositions}
In contrast with the theory of finite dimensional Lie algebras of characteristic $0$, there is no Levi decomposition result for infinite dimensional ones. However, in case a graded Lie algebra has finite depth (e.g., if it has finite cohomological dimension), the beginnings of a structure theory still exist (see \cite{radical}). 

Let $\li$ be a graded Lie algebra and consider the sum $R=\rad\li$ of all the solvable ideals of $\li$; this ideal is called the \textbf{radical} of $\li$ and it must not be solvable. If depth of $\li$ is finite, then $\li/R$ admits an essentially-semisimple decomposition, i.e., $\li/R$ splits into the direct product of  finitely-many non-abelian ideals $I(1),\dots,I(r)$ ($r\geq 0$) satisfying the following property: If $L,M\subseteq I(j)$ are non-zero ideals of $\li/R$, then $L\cap M\neq0$.
 If $\li$ has finite dimension (and it is not necessarily graded), then $R$ is solvable, the quotient $\li/R$ is semisimple and all the $I(r)$ are simple ideals of $\li/R$.

\begin{lem}\label{lem:zinfinity}
	Let $\li$ be a graded Lie algebra. Then, every finite dimensional ideal of $\li$ is contained in the limit space $Z_{\infty}(\li)$ of the upper central series.
	\begin{proof}
		Let $Z_1=Z(\li)$ and $Z_{n+1}=\set{x\in \li}{[x,\li]\subseteq Z_n}$ be the terms of the upper central series. 
		
		Fix a finite dimensional ideal $I$ of $\li$.
		If $x\in I$ is an element of maximal degree $M$ in $I$, then $[x,\li]=0$, i.e., $x\in Z_1$. If we assume by induction that $I_{M-i}\subseteq Z_{i+1}$, then, for $x\in I_{M-i-1}$, one has $[x,\li]\subseteq I_{\geq M-i}\subseteq Z_{i+1}$ and hence $x\in Z_{i+2}$.
	\end{proof}
\end{lem}
\begin{cor}\label{cor:rad=zinf}
	Let $\li$ be a locally-finite graded Lie algebra of finite depth. Then, $\rad\li =Z_{\infty}(\li)$. In particular, 
	\begin{enumerate}
		\item The non-zero homogeneous component of maximal degree of $\rad\li$ is contained in the center, and 
		\item $\rad(\li)=0$ if $\li$ is centerless.
	\end{enumerate}
	\begin{proof}
		The radical $\rad\li$ is a finite dimensional ideal by \cite{radical}, and hence, by Lemma \ref{lem:zinfinity}, $\rad\li\subseteq Z_{\infty}(\li)$. On the other hand, $Z_{n}(\li)$ is a solvable ideal of $\li$, for all $n\geq 1$, and hence $Z_{\infty}(\li)\subseteq \rad(\li)$.
	\end{proof}
\end{cor}

If $\li$ is a BK Lie algebra, or it admits a quadratic filtration, then, by Theorem \ref{thm:centerBK}, the center of $\li$ is concentrated in degree $1$ and $\li\simeq Z(\li)\times \li/Z(\li)$. By Corollary \ref{cor:rad=zinf}, the radical of $\li$ coincides with the center, and hence the direct factor $\li/Z(\li)$ of $\li$ is essentially-semisimple.

\begin{lem}\label{lem:sumBK}
	The direct product of non-abelian graded Lie algebras is not BK. 
		\begin{proof}
			Let $\mc A$ and $\mc B$ be non-abelian graded Lie algebras and let $\li=\mc A\times\mc B$ be their free product. By contradiction, assume that $\li$ is BK, and hence so are $\mc A$ and $\mc B$. 
			
			Since $\mc A$ and $\mc B$ are non-abelian, there exist elements $a,a'\in \mc A$ and $b,b'\in \mc B$ of degree $1$ such that $[a,a']\neq 0\neq [b,b']$. 
			
			In particular, the two subalgebras generated by $a,a'$ and by $b,b'$ are isomorphic with the free Lie algebra $\mc F$ of rank $2$. The Lie algebra generated by the four elements $a,a',b,b'\in \li$ is thus isomorphic to $\mc F\times \mc F$. 
			
			For $t=a'+b'$, we claim that the subalgebra $\mc M$ of $\li$ generated by $a,b$ and $t$ is not quadratic, concluding the proof. 
			
			 Indeed, in $\mc M$ one has the minimal relation $[[t,a],b]=0$.
			 
			 Notice that the standard Lie algebra $\mc F\times \mc F$ is isomorphic to the RAAG Lie algebra associated to the square graph, and hence it is not BK by Example \ref{ex:raag}
		\end{proof}
\end{lem}
We deduce that BK Lie algebras satisfy an \textit{essential} version of the Levi decomposition theorem in a very specific way.
\begin{thm}
	Let $\li$ be a non-abelian BK Lie algebra. Then, there exists an abelian Lie algebra $Z$ and an essentially simple ideal $\mc M$ of $\li$ such that $\mc M$ is a BK Lie algebra and $\li=Z\times \mc M$.
	\begin{proof}
		As noted above, $\li$ is the direct sum of its center and an essentially-semisimple Lie algebra $\mc M=I(1)\times \dots\times I(r)$, where $r\geq 1$ and $I(j)$ is an essentially simple ideal of $\li$. Since $\mc M$ is BK, the number $r$ of components $I(j)$ must be $1$, by Lemma \ref{lem:sumBK}.
	\end{proof}
\end{thm}

\section{Examples}\label{sec:g2d}
\subsection{Surface Lie algebras}
If $d\geq 1$, one defines the surface Lie $k$-algebra \[\mc G_{2d}=\pres{x_1,y_1,\dots,x_d,y_d}{\sum_i[x_i,y_i]}.\]

It is the associated graded Lie $k$-algebra of the oriented surface group of genus $d$ with respect to the lower central series. If the ground field is $\F_p$ with $p$ odd, the $p$-restrictification of $\mc G_{2d}$ is the Lie algebra associated to the dimension subgroup series of any $2d$-generated Demu\v skin group \cite{MPPT}.

In \cite{sb}, the author proved that $\mc G_{2d}$ is BK and all of its proper subalgebras are free. The usage of HNN-extensions allows us to prove that the converse holds. 

\begin{prop}\label{prop:charG2d}
	Let $k$ be a field of characteristic $\neq 2$. Let $\li$ be a quadratic Lie $k$-algebra such that all of its proper standard subalgebras are free. Then, either $\li$ is free or $\li$ is a surface Lie algebra. 
\end{prop}

Since $\mc G_{2d}$ is Koszul, its cohomology ring can be easily computed:

\[H^\bu(\mc G_{2d})=\exa(V)/(\Omega)\]
where $V=\spn\{\xi_1,\eta_1,\dots,\xi_d,\eta_d\}$ and $\Omega=\spn\{\xi_i\wedge\eta_j-\delta_{ij}\xi_1\wedge\eta_1,\xi_i\wedge \xi_j,\eta_i\wedge \eta_j\}$.
\begin{lem}\label{lem:A3=0}
	Let $A$ be a connected graded commutative algebra of characteristic $\neq 2$ that is generated by elements of degree $1$. If $A_3=0$ and $\dim A_2=1$, then $A$ is the direct sum of an exterior algebra and the cohomology of a surface Lie algebra. 
	\begin{proof}
	The multiplication map restricts to a skew-symmetric bilinear form \[\beta:A_1\times A_1\to A_2=k.\] Consider the radical of $\beta$, \[R=\rad\beta=\set{a\in A_1}{\beta(a,b)= 0,\ \forall b\in A_1}\]  and let $C_1\subseteq A_1$ be a complement of the $k$-subspace $R$ in $A_1$, i.e. $A_1=R\oplus C_1$.
	 The restriction of $\beta$ to $C_1$ defines a symplectic form. Consider also the subalgebra $B$ of $A$ generated by $R$. Then, \begin{align*}
			B&=k\cdot 1+R, \\
			C&=k\cdot 1+ C_1+ A_2.
		\end{align*}

		Since $R\cdot C_1=R\cdot A_2=0$, it follows that $A=B\sqcap C$, i.e., $A$ is the connected algebra with $A_i=B_i\times C_i$ ($i\geq 1$) and $B\cdot C=0$. 
		One can find a Darboux basis $x_1,y_1,\dots,x_d,y_d$ for the symplectic space $(C_1,\beta_C)$, i.e. a basis of $C_1$ such that, for the symplectic form $\beta_C:C_1\otimes C_1\to A_2=k$, it holds \begin{align*}
			&\beta_C(x_i\otimes y_j)=\delta_{ij},\\
			&\beta_C(x_i\otimes x_j)=\beta(y_i\otimes y_j)=0.
		\end{align*}
		Now, both $B$ and $C$ are quadratic algebras, given by the quotients \begin{align*}
				&C=\frac{\exa(C_1)}{(x_i\wedge y_j-\delta_{ij}x_1\wedge y_1, x_i\wedge x_j,y_i\wedge y_j)}\\
			&\ \qquad\qquad\qquad B=\frac{T(R)}{T^2(R)}
		\end{align*}
		It follows that $B=H^\bu(\mc F)$, where $\mc F$ is the standard free Lie algebra generated by $R^\ast$, and $C\simeq H^\bu(\mc G_{2d})$. 
	\end{proof}
\end{lem}

		\begin{proof}[Proof of Prop. \ref{prop:charG2d}]
		Assume $\li$ has more than $1$ minimal relation. Then, there is a quadratic Lie algebra $\tilde\li$ with $2$ relations, and an epimorphism $\pi:\tilde\li\to\li$ that is an isomorphism in degree $1$. 
		If $\mc M$ is a proper standard subalgebra of $\tilde\li$, then its image under $\pi$ is a proper standard subalgebra of $\li$, which needs to be free by hypothesis. In particular, $\mc M$ is free. Indeed, if $\mc F$ is the free Lie algebra on $\mc M_1$, then the composition $\mc F\to \mc M\to \pi(\mc M)$ is an isomorphism.
		
		Hence, $\tilde\li$ satisfies the hypothesis of the theorem.
		Thus, assume $\li$ to be defined by two (non-necessarily independent) quadratic relations $s\neq 0$ and $r$. Let $\mc G$ be the $1$-relator Lie algebra defined by $s$ that covers $\li$. 
		
		Since $\mc G$ is $1$-generated, there is a non-zero element $x\in \mc G_1$, a maximal standard subalgebra $\mc M$ of $\mc G$ not containing $x$, and elements $z\in \mc M_1$ and $c\in \mc M_2$ such that \[r=[x,z]+c.\]
		
		By setting $\phi(z)=-c$, one has the well defined derivation $\phi:\gen z\to \mc M$, and the HNN decomposition  \[\li\simeq\hnn_\phi(\mc M,t)\] where $t$ corresponds to $x$ in the isomorphism. Since $\mc M$ embeds into $\li$, it is free, whence the Lie algebra $\li$ has just one relation, $[t,z]=-c$.
		
		Now, since $\mc M$ is free, both $\mc M$ and $\mc A$ are Koszul, and hence so is $\li$. As $\li$ is a Koszul $1$-relator Lie algebra, its cohomology ring satisfies the hypothesis of Lemma \ref{lem:A3=0}, proving that $\li$ is a free product of a free Lie algebra $\mc F$ and a surface Lie algebra $\mc G_{2d}$.
		
		 If $\mc F\neq 0$, then $\mc G_{2d}$ is a non-free subalgebra of $\li$, contradicting the hypotheses.
	\end{proof}
	
	As a corollary, all the BK Lie algebras that are not free have a subalgebra isomorphic to $\mc G_{2d}$ for some $d\geq 1$. 
	
	\subsection{Quadratic $2$-relator Lie algebras}
	  Here we want to present another large class of such Lie algebras, motivated by Quadrelli \cite{2rel}. 
	If $r$ is an integer, we say that a graded Lie algebra $\li$ is $r$-relator if $b_2(\li)=r$, i.e., there exists a free Lie algebra $\mc F$ and $r$ homogeneous elements minimally generating an ideal $R$ of $\mc F$ such that $\li=\mc F/R$.
	
	\begin{thm}\label{thm:2rel}
		If $\li$ is a $2$-relator quadratic Lie algebra, then $\li$ is BK.
		\begin{proof}
			Let $A$ be the diagonal cohomology of $\li$, i.e., $A=\bigoplus_iH^{i,i}(\li)$. By definition and \cite{pp}, $A$ is a quadratic graded-commutative algebra with $\dim A_2=2$. 
			Let $x$ be any non-zero element of degree $1$. 
			
			If $xA_1=0$ or $xA_1=A_2$, then $xA_2=0$. On the other hand, if $xA_1$ is a $1$-dimensional vector space, say generated by an element $a\neq 0$, then there is a basis $z_1=x,\ z_2=y,\ z_3,\dots,z_n$ for $A_1$ such that $xy=a$ and $xz_i=0$ for $i\geq 3$. Let $b\in A_2$ so that $A_2=\spn\{a,b\}$. Let $\alpha_{ij}\in k$ be the coefficient of $z_iz_j$ in some expression of $b$. Since $xz_j=0$ for $j>2$, we can take $\alpha_{1j}=0$ for $j>2$, and, up to replacing $b$ with $b-\alpha_{12}a$, one can suppose that $\alpha_{12}=0$. In particular, $b$ lies in the subalgebra generated by the $z_i$'s, for $i> 2$, and hence $bx=0$.
			
			This proves that $A_3=0$. By \cite[Thm. 3.7]{2rel}, it follows that $A$ is universally Koszul. In particular, since a standard Lie algebra is Koszul precisely when its diagonal cohomology is Koszul, $\li$ is Koszul and $A=H^\bu(\li)$. On the other hand, $A$ is universally Koszul, which proves that $\li$ is BK.
		\end{proof}
	\end{thm}
	
	As $1$-relator quadratic Lie algebras are BK, we deduce that a quadratic $r$-relator Lie algebra $\li$ is BK if $r\leq 2$. Since for a quadratic Lie algebra $b_2(\li)\leq {{b_1(\li)}\choose 2}$ and the equality holds iff $\li$ is abelian, all quadratic Lie algebras generated by at most $3$ elements are BK. In fact, the first example of a quadratic non BK Lie algebra is generated by $4$ elements (e.g., the RAAG associated to a square graph, and the Heisenberg Lie algebra $\mf h_2$, which is not even Koszul).
	
\subsection{Further examples}
Define \[\mc B_{2d}=\pres{x_1,y_1,\dots,x_d,y_d}{[x_i,y_i]-[x_1,y_1]:\ i=2,\dots,d}.\]

Observe that $\mc B_{4}\simeq \mc G_{4}$.

\begin{lem}
	For all $d\geq 1$, $\mc B_{2d}$ is Koszul and locally of type FP.
	\begin{proof}
		For $d=1$, we have $\mc B_{2}=k^2[-1]$. For $d=2$, $\mc B_{4}=\mc G_{4}$ is the surface Lie algebra. 
		
		Now let $d>2$. Consider $\li=\mc B_{2(d-1)}\amalg\gen{y_n}$. By induction, $\li$ is Koszul and $\li'$ is free. 
		
		Also, $\phi:\mc A=\gen{y_n}\to \li$ defined by $\phi(y_n)=[x_1,y_1]$ is a derivation. Hence, $\mc B_{2d}=\hnn_\phi(\li,x_n)$ is Koszul. Moreover, $\mc B_{2d}'\cap\mc A=0$ and $\mc B_{2d}'\cap\li\subseteq \li'$ is free, proving that $\mc B_{2d}'$ is free as well by \cite[Thm. 3]{hnnLS}. By \cite[Prop. 5.8]{cmp}, we deduce that $\mc B_{2d}$ is locally of type FP.
	\end{proof}
\end{lem}
%
Since $\mc B_{2d}$ is Koszul, its cohomology algebra is \[\mc B_{2d}^!=\pres{\xi_1,\eta_1,\dots,\xi_d,\eta_d}{\xi_i\eta_j(1-\delta_{ij}),\ \xi_i\xi_j,\ \eta_i\eta_j,\ \sum_i\xi_i\eta_i}.\] In particular, $\cd\mc B_{2d}=2$. 
Notice that $(\mc B_{2d}^!)_2=(\xi_1+\dots+\xi_d)\cdot (\mc B_{2d}^!)_1$, i.e., $\frk\mc B_{2d}=\dim(\mc B_{2d})_1-\cd\mc B_{2d}+1=2d-1$.

\begin{thm}
	The Lie algebra $\mc B_{2d}$ is BK.
	\begin{proof}
		Let $A$ be the cohomology algebra of $\mc B_{2d}$ with the presentation as above. By \cite[Prop. 20]{MPQT}, for proving that $A$ is universally Koszul, we need to show that $(I:b)$ is a $1$-generated ideal of $A$ for every $1$-generated ideal $I$ of $A$ and $b\in A_1\setminus I$. 
		
		So, let $I$ be an ideal of $A$ different from $A_+$ and let $b\in A_1\setminus I$. 
		
		Since $A_3=0$, one has $A_2\subseteq (I:b)$.
		Denote by $F$ the ideal of $A$ generated by $(I:b)_1$. We claim that $F$ contains the elements $\xi_i\eta_i$ for all $i$.
		
		Set $b=\sum_j \alpha_j\xi_j+\beta_j\eta_j$. 
		
		Fix an index $i\in\{1,\dots,d\}$.
		
		If $\alpha_i=0$, then $b\eta_i=0$ and hence $\eta_i\in (I:b)$, so that $\xi_i\eta_i\in F$.
		
		Let $\alpha_i\neq 0$. Since $b\in (I:b)_1$, we get $\xi_i\eta_i=b\cdot(\alpha_i^{-1}\eta_i)\in F$.
	\end{proof}
\end{thm}

%

\bibliography{mybibtex.bib}

\end{document}